\documentclass{amsart}
\usepackage{amssymb,amsthm,amstext,latexsym,amsmath,fullpage}
\usepackage{enumitem}
\usepackage{amsaddr}
\usepackage[all]{xy}

\newcommand{\B}{\mathrm{B}}                
\newcommand{\E}{\mathrm{E}}                
\newcommand{\Lin}{\mathrm{Lin}}            
\newcommand{\Hh}{\mathrm{H}}               
\newcommand{\R}{\mathbb{R}}                
\newcommand{\T}{\mathrm{T}}                
\newcommand{\cotan}{\mathrm{T}^*}          
\newcommand{\Z}{\mathbb{Z}}                

\newcommand{\aff}{\mathrm{Aff}}            
\newcommand{\affz}{\mathrm{Aff}_{\Z}}      
\newcommand{\gl}{\mathrm{GL}}              
\newcommand{\glnr}{\mathrm{GL}(n;\R)}      
\newcommand{\glnz}{\mathrm{GL}(n;\Z)}      
\newcommand{\pib}{\pi_1(B)}                

\newcommand{\am}{(B,\mathcal{A})}          
\newcommand{\aam}{(N,\mathcal{A})}         
\newcommand{\baff}{\mathrm{B}\aff(\rzn)}   
\newcommand{\bgl}{\mathrm{B}\glnz}         
\newcommand{\brzn}{\mathrm{B}\rzn}         
\newcommand{\rzn}{\R^n/\Z^n}               
\newcommand{\sm}{(M,\omega)}               
\newcommand{\tn}{\mathbb{T}^n}             

\newcommand{\alf}{\rzn \hookrightarrow M \to B}                 
\newcommand{\flf}{F \hookrightarrow \sm \to B}                  
\newcommand{\lf}{\rzn \hookrightarrow (M,\omega) \to B}         
\newcommand{\rlf}{\rzn \hookrightarrow (\cotan B/P_{\am},\omega_0) \to B}    
\newcommand{\plb}{\Z^n \hookrightarrow P \to B}                 
\newcommand{\univcover}{\xymatrix@1{\glnz\,\ar@{^{(}->}[r]       
    &\mathrm{B}(\rzn) \ar[r]^-{\tau} & \mathrm{B}\affz(\rzn)}}
\newcommand{\univlf}{\xymatrix@1{\rzn\,\ar@{^{(}->}[r]
    &\mathrm{B}\glnz \ar[r]^-{\sigma} & \mathrm{B}\affz(\rzn)}}
\newcommand{\ulf}{\rzn \hookrightarrow \baff \to \bgl}
\newcommand{\univplb}{\mathrm{H}_1(\rzn;\Z) \hookrightarrow
    P_n \to \mathrm{B}\affz(\rzn)}

\newcommand{\de}{\mathrm{d}}               
\newcommand{\detwo}{\de^{(2)}}             
\theoremstyle{definition}
\newtheorem{defn}{Definition}[section]
\newtheorem{rk}[defn]{Remark}

\newtheorem*{mr}{Main Result}
\newtheorem*{qn}{Question}
\newtheorem{ex}[defn]{Example}

\theoremstyle{remark}
\newtheorem*{notation}{Notation}

\theoremstyle{plain}
\newtheorem{thm}{Theorem}[section]
\newtheorem{lemma}[thm]{Lemma}
\newtheorem{cor}[thm]{Corollary}
\newtheorem{prop}[thm]{Proposition}

\title{Universal Lagrangian Bundles}
\author{Daniele Sepe}
\thanks{The author would like to thank Jos\'e Figueroa-O'Farrill and Jarek
  Kedra for insightful comments on earlier versions of this
  article, and the anonymous referee for pointing out how to improve a previous version of the paper.}
\address{Department of Mathematics, University of Leicester,
  University Road, Leicester LE1 7RH, United Kingdom}
\email{dsepe@math.ist.utl.pt}
\begin{document}
\maketitle
\begin{abstract}
The obstruction to construct a Lagrangian bundle over a fixed integral affine manifold was constructed by Dazord and Delzant in \cite{daz_delz} and shown to be given by `twisted' cup products in \cite{sepe_lag}. This paper uses the topology of universal Lagrangian bundles, which classify
Lagrangian bundles topologically (cf. \cite{sepe_topc}), to reinterpret this obstruction as the vanishing of a differential on the second page of a Leray-Serre spectral sequence. Using this interpretation, it
is shown that the obstruction of Dazord and Delzant depends on an important cohomological
invariant of the integral affine structure on the base space, called
the radiance obstruction, which was introduced by Goldman and Hirsch
in \cite{goldman}. Some examples, related to non-degenerate
singularities of completely integrable Hamiltonian systems, are discussed. 
\end{abstract}
\setcounter{tocdepth}{1}
\tableofcontents

\section{Introduction} \label{sec:introduction}
A fibre bundle $F \hookrightarrow \sm \to B$ is \emph{Lagrangian} if $\sm$ is
a $2n$-dimensional symplectic manifold and the fibres are Lagrangian
submanifolds, \textit{i.e.} $w|_{F} =0$ and $\dim F = n$. Throughout
this paper, the fibres are assumed to be compact and
connected. Lagrangian bundles arise naturally in many fields of
mathematics, ranging from classical and quantum Hamiltonian mechanics
to mirror symmetry, as the regular parts of the so-called Lagrangian
fibrations. These admit singular fibres and are one of the geometric
constructs used to study the topology, geometry and dynamics of
completely integrable Hamiltonian systems 
(cf. Appendix D in \cite{cush_dav}). Lagrangian bundles contain information about
singularities of completely integrable Hamiltonian 
systems and are, therefore, central to the study of these dynamical
systems. Under the above 
assumptions on the fibres, the Liouville-Mineur-Arnol`d theorem
(cf. Section 49 in \cite{arnold}) implies that the fibres of a Lagrangian bundle are diffeomorphic to
$n$-dimensional tori $\tn$, and that a neighbourhood of a fibre is
symplectomorphic to a neighbourhood of the zero section of the
cotangent bundle $(\cotan B, \Omega_{\tn}) \to \tn$, where
$\Omega_{\tn}$ is the canonical symplectic form on $\cotan \tn$. This
theorem gives semi-global (\textit{i.e.} in the neighbourhood of  a
fibre) topological and symplectic 
classifications of Lagrangian bundles. The global topological classification
has been achieved by Duistermaat in
\cite{dui}; there are only two topological invariants, namely the
topological \emph{monodromy} and the \emph{Chern class}, which is the
obstruction to the existence of a section. Dazord and Delzant extended
these notions to the broader class of isotropic bundles (whose
fibres are not necessarily of maximal dimension) in \cite{daz_delz};
moreover, they introduced a symplectic invariant, the
\emph{Lagrangian} class, which is the obstruction to the existence of
a Lagrangian section. \\ 

Closely related to the classification problem is the question of constructing
Lagrangian bundles over a fixed manifold $B$. The
Liouville-Mineur-Arnol'd theorem, along with an observation due to
Markus and Meyer in \cite{markus_meyer} and Duistermaat \cite{dui},
provides a necessary condition for a $\tn$-bundle to be Lagrangian,
namely that the fibres admit a smoothly varying \emph{affine
  structure}, which is a $\mathrm{C}^{\infty}$-atlas $\mathcal{A}$
whose changes of coordinates are constant affine transformations of
$\R^n$. Such bundles are called \emph{affine} $\rzn$-bundles. The
above condition is not sufficient
(cf. \cite{baier,sepe_topc}), as the total space does not necessarily
admit an appropriate symplectic form. The existence of such a form on
the total space of $\tn \hookrightarrow \sm \to B$ implies that the
base space is an \emph{integral affine manifold}, \textit{i.e.} an
affine manifold whose atlas $\mathcal{A}$ has coordinate changes lying
in
$$ \aff_{\Z} (\R^n) := \glnz \ltimes \R^n. $$
\noindent
This property restricts the topology of manifolds that arise as the
base space of Lagrangian bundles: for instance, the only closed
orientable (integral) affine surface is $\mathbb{T}^2$
(cf. \cite{benzecri,milnor_euler}). The above observation, originally
due to Duistermaat in \cite{dui}, entwines the construction problem of
Lagrangian bundles with the study of affine differential geometry,
which was first introduced in the '50s and has since attracted much
attention (\textit{e.g.}
\cite{aus,hirsch,goldman,goldman_orbits,milnor_euler,smillie}). \\

Let $\tn \hookrightarrow \sm \to B$ be a Lagrangian bundle, denote by
$\chi_*: \pi_1(B;b) \to \glnz$ its monodromy (cf. Definition
\ref{defn:monodromy}) and let $\mathcal{A}$ be the induced integral
affine atlas on $B$. Associated to $\mathcal{A}$ is a flat,
torsion-free connection $\nabla$ on the tangent bundle $\T B$
(cf. \cite{aus}); denote by
$$ \mathfrak{l} : \pi_1(B;b) \to \glnz $$
\noindent
its linear holonomy representation. Then 
\begin{equation}
  \label{eq:1}
  \chi_* = \mathfrak{l}^{-T},
\end{equation}
\noindent
where $-T$ denotes inverse transposed. Therefore, in order to study
the construction problem for Lagrangian bundles over $B$, it is necessary to
fix an integral affine structure $\mathcal{A}$ on $B$ and consider affine
$\rzn$-bundles over $\am$ whose monodromy satisfies equation
(\ref{eq:1}). Such bundles are called \emph{almost Lagrangian}
(cf. Definition \ref{defn:almost_lag_bundles}). The
isomorphism class of these bundles depends on a cohomology class
$$ c \in \Hh^2(B;\Z^n_{\chi_*}),$$
\noindent
\textit{i.e.} the obstruction to the existence of a section. A result
due to Dazord and Delzant in 
\cite{daz_delz} shows that there is a homomorphism
$$ \mathcal{D}_{\am} : \Hh^2(B;\Z^n_{\chi_*}) \to \Hh^3(B;\R) $$
\noindent
whose kernel gives the subgroup of classes whose corresponding almost
Lagrangian bundles are Lagrangian. This homomorphism can be computed explicitly and is obtained by taking a `twisted' cup product on $B$, as shown in \cite{sepe_lag}. However, the work of \cite{sepe_lag} leaves some natural questions regarding $\mathcal{D}_{\am}$ unanswered, namely

\begin{itemize}
\item given that $\mathcal{D}_{\am}$ measures, in some sense, the obstruction to constructing a suitable symplectic form on the total space of an almost Lagrangian bundle, can it be related to some differentials in the Leray-Serre spectral sequence of said bundle?
\item what is the exact relation between $\mathcal{D}_{\am}$ and the integral affine geometry of $\am$?
\item the twisted cup product of Definition 4 in \cite{sepe_lag} does not come from a pairing between $\Hh^2(B;\Z^n_{\chi_*})$ and some other cohomology group on $B$. Can $\mathcal{D}_{\am}$ be understood as an `honest' twisted cup product, \textit{i.e.} for any $c \in \Hh^2(B;\Z^n_{\chi_*})$, is $\mathcal{D}_{\am}(c)$ obtained by pairing $c$ with some suitable fixed cohomology class?  
\end{itemize}

The main novelty of the paper is that it provides an answer to the above questions, thus providing a more natural, geometric description of $\mathcal{D}_{\am}$. In particular, the main result of this paper shows that $\mathcal{D}_{\am}$ is determined by an integral affine invariant
of $\am$, namely the \emph{radiance obstruction} 
$$r_{\am} \in \Hh^1(B;\R^n_{\mathfrak{l}}). $$
\noindent
This cohomology class has been defined by Goldman and Hirsch in
\cite{goldman} (and implicitly by Smillie in \cite{smillie}); it is
the obstruction to homotoping the coordinate changes in
$\mathcal{A}$ to take values in $\glnz$. In fact, the main result of
the paper can be stated as follows.

\begin{mr}
  The obstruction for an almost Lagrangian bundle over an integral
  affine manifold $\am$ with
  linear holonomy $\mathfrak{l}$ to be Lagrangian is given by 
  $$ \mathcal{D}_{\am}(c) = c \cdot r_{\am} \in \Hh^3(B;\R), $$
  \noindent
  where $\cdot$ denotes the pairing in cohomology 
  $$ \Hh^2(B;\Z^n_{\chi_*}) \otimes_{\Z} \Hh^1(B;\R^n_{\mathfrak{l}}) \to \Hh^3(B;\R) $$
  \noindent
  obtained by combining cup product on $B$ with the duality between the coefficient systems.
\end{mr}

Note that the above result implies Theorem 3 in \cite{sepe_lag} and it also explains the geometric reason for why the `twisted' cup product of \cite{sepe_lag} appears there in the first place (cf. Remark \ref{rk:relation_alo}). \\

The Main Result is achieved by combining Theorems \ref{thm:diff_general}, \ref{thm:realisability} and \ref{thm:equality_forms} below. In particular, its proof can be broken down into the following three steps, each of which emphasises a new aspect of the Dazord-Delzant homomorphism.

\subsection*{Step 1}
The Leray-Serre spectral sequence of \emph{universal Lagrangian bundles}
(cf. \cite{sepe_topc}) is studied. These bundles classify, up to isomorphism,
the topological 
type of affine $\rzn$-bundles and, in particular, of those which are
Lagrangian. Each universal Lagrangian bundle has two topological
invariants, namely its topological monodromy and the universal Chern
class $c_U$. Using techniques from the work of Charlap and Vasquez
\cite{charlap}, Theorem \ref{thm:diff_general} proves that some
differential $\detwo$ 
on the $\E_2$-page of the Leray-Serre spectral sequence with
$\Z$-coefficients of the universal Lagrangian bundle is
given by taking cup products with the universal Chern class
$c_U$. This is a non-trivial generalisation of the computation of the
corresponding differential for the universal principal $\tn$-bundle
(cf. Lemma \ref{lemma:main_result_torus}). 

\subsection*{Step 2}
Associated to an integral affine manifold $\am$ is the
\emph{symplectic reference bundle} 
$$\tn \hookrightarrow (\cotan
B/P_{\am}, \omega_0) \to B, $$
\noindent
which admits a globally defined
Lagrangian section and it induces the given integral affine atlas
$\mathcal{A}$ on $B$. The symplectic form $\omega_0$ defines a
cohomology class
$$ w_0 \in \Hh^1(B;\Hh^1(\tn;\R)_{\mathfrak{l}}), $$
\noindent
(cf. Lemma \ref{lemma:coho_class_omega_0}). Using Step 1, the obstruction for an
almost Lagrangian bundle with Chern class $c$ to be Lagrangian is
proved to be given by
$$ \detwo (w_0) = c \cdot w_0, $$
\noindent
up to some isomorphisms (cf. Theorem \ref{thm:realisability}). In light of the results in \cite{daz_delz}, this proves that the Dazord-Delzant homomorphism comes from the differential on the second page of the Leray-Serre spectral sequence of the underlying fibre bundle.

\subsection*{Step 3} The cohomology class $w_0$ is shown to be mapped
to the radiance obstruction $r_{\am}$ under an isomorphism induced by
the symplectic form $\omega_0$ (cf. Theorem
\ref{thm:equality_forms}). Combining this result with Steps 1 and 2 above, obtain the Main Result, which expresses the Dazord-Delzant homomorphism as the pairing $\cdot$ with the radiance obstruction of $\am$.\\

The structure of the paper is as follows. Section
\ref{sec:defin-basic-prop} constructs the universal Lagrangian
bundles of \cite{sepe_topc} and the corresponding topological
invariants, namely the monodromy and universal Chern class. These
results only depend on the existence of a smoothly varying affine
structure on the fibres of Lagrangian bundles. The importance of the
integral affine geometry of the base space of a Lagrangian bundle is
explained in Section \ref{sec:construction-problem}, where the
construction problem is posed and almost Lagrangian bundles are
defined. Throughout this section, the connection with the sheaf
theoretic point of view on the problem (cf. \cite{daz_delz,dui}) is
highlighted. Section \ref{sec:spectral-sequence} brings universal
Lagrangian bundles into the problem of constructing Lagrangian bundles
over an integral affine manifold $\am$; Theorem
\ref{thm:realisability} proves that the homomorphism
$\mathcal{D}_{\am}$ of Dazord and Delzant is given by taking twisted
cup products with $w_0$, the cohomology class of the symplectic form
of the symplectic reference bundle. This is obtained via a study of
the Leray-Serre spectral sequence of universal Lagrangian bundles
(cf. Theorem \ref{thm:diff_general}). Finally, Section
\ref{sec:relat-integr-affine} defines the radiance obstruction
$r_{\am}$ of an (integral) affine manifold $\am$ and completes the
proof of the Main Result stated above. Moreover, it exploits the
connection between the symplectic topology of Lagrangian bundles and
integral affine geometry in Theorem
\ref{thm:no_closed_integral_affine_mflds} and by considering some
explicit examples which are related to singular Lagrangian bundles.
 

\section{Universal Lagrangian
  Bundles and Topological Invariants} \label{sec:defin-basic-prop}
\subsection{Definition of universal Lagrangian
  bundles}\label{sec:struct-group-lagr}
\subsubsection*{Structure group}The Liouville-Mineur-Arnol'd theorem,
along with a crucial observation
due to Markus and Meyer \cite{markus_meyer} and Duistermaat
\cite{dui}, can be used to prove the following theorem, stated below
without proof, regarding the structure of Lagrangian bundles.

\begin{thm}[Markus and Meyer \cite{markus_meyer}, Duistermaat
  \cite{dui}] \label{thm:existence_of_plb} 
  Let $\flf$ be a Lagrangian bundle. Then
  \begin{enumerate}[label=\roman*), ref=\roman*]
  \item \label{item:3} for each $b \in B$, the fibre $F_b =
    \pi^{-1}(b)$ is diffeomorphic to $\tn$;
  \item \label{item:4} the structure group of the fibre bundle reduces
    to
    $$ \glnz \ltimes \rzn. $$
  \end{enumerate}
\end{thm}

One of the consequences of Theorem \ref{thm:existence_of_plb} is that
the fibres of a Lagrangian bundle are naturally \emph{affine
  manifolds}.

\begin{defn}[Affine manifolds] \label{defn:affine_manifold}
  An \emph{affine structure} on an $n$-dimensional manifold $B$ is a
  choice of atlas $\mathcal{A} = \{(U_{\alpha},
  \phi_{\alpha}:U_{\alpha} \to \R^n)\}$ whose changes of coordinates
  
  $$ \phi_{\beta} \circ \phi^{-1}_{\alpha}: \phi_{\alpha}(U_{\alpha}
  \cap U_{\beta}) \subset \R^n \to \phi_{\beta}(U_{\alpha}
  \cap U_{\beta}) \subset \R^n$$
  \noindent
  are constant on connected components, and are \emph{affine
    transformations} of $\R^n$, \textit{i.e.} they lie 
  in the group 
  $$ \aff (\R^n) := \gl(n;\R) \ltimes \R^n, $$
  \noindent
  where the action of $\gl(n;\R)$ on $\R^n$ is the standard one. An
  \emph{affine manifold} is a pair $\am$, where $B$ is a
  manifold and $\mathcal{A}$ is an affine structure on $B$. 
\end{defn}

In particular, the fibres of a Lagrangian bundle can be smoothly
identified with the following affine manifold.

\begin{ex}[Standard affine structure on $\tn$]\label{ex:std_affine_rzn}
  Let $\Lambda \subset (\R^n,+)$ be the standard cocompact
  lattice and let $\mathbf{e}^1,\ldots,\mathbf{e}^n$ denote the
  standard generators of $\Lambda$. Define a $\Lambda$-action on
  $\R^n$ by
  $$ \mathbf{a} \cdot \mathbf{e}^i= \mathbf{a} + \mathbf{e}^i, $$
  \noindent
  where $\mathbf{a} = (a^1,\ldots,a^n)$ denotes the standard (hence affine)
  coordinates on $\R^n$. The quotient $\R^n/\Lambda$
  inherits an affine structure $\mathcal{A}$,
  since the above action is by \emph{affine 
    diffeomorphisms} of $\R^n$, \textit{i.e.} diffeomorphisms which
  are affine in local affine coordinates. For notational ease, this
  affine manifold is henceforth denoted by $\rzn$.
\end{ex}

\begin{rk}
  Theorem \ref{thm:existence_of_plb} implies that the structure group of
  a Lagrangian bundle can be reduced to the group of affine
  diffeomorphisms of $\rzn$, denoted by $\aff(\rzn)$. 
\end{rk}

\begin{notation}
  The fibres of a Lagrangian bundle are henceforth denoted by $\rzn$
  to emphasise the importance of their affine structure. Affine
  coordinates on $\rzn$ are denoted by $\boldsymbol{\theta}$.
\end{notation}

\begin{rk}\label{rk:affine_structure_translations}
  The subgroup of $\aff(\rzn)$ consisting of translations can be
  identified with the affine manifold $\rzn$ via the map that takes a
  translation $\mathcal{T}$ to $\mathcal{T}\boldsymbol{\theta}_0$, where
  $\boldsymbol{\theta}_0 \in \rzn$ is any fixed point. 
\end{rk}

The existence of a smoothly varying affine structure on the fibres of
a Lagrangian bundle $\lf$ implies the existence of an associated $\Z^n$-bundle
$\plb$ whose fibres are smoothly identified with the lattice $\Lambda$
of Example \ref{ex:std_affine_rzn}.

\begin{defn}[Period lattice bundle \cite{dui}]\label{defn:plb}
  The covering space $\plb$ is called the \emph{period lattice bundle}
  associated to $\lf$.
\end{defn}

\subsubsection*{Definition of universal Lagrangian bundles}
The topology of the group $\aff(\rzn)$ can be used to construct
universal Lagrangian bundles. Let $0$ and $1$ denote the trivial group
with additive and multiplicative structures respectively. There is a
short exact sequence of groups
\begin{equation}
  \label{eq:16}
  \xymatrix@1{0 \ar[r] & \rzn \ar[r]^-{\tau} & \aff(\rzn)
    \ar[r]^-{\mathrm{Lin}} & \glnz \ar[r] & 1,}
\end{equation}
\noindent
where the homomorphisms $\tau, p$ are defined by
\begin{equation*}
  \begin{split}
    \tau: \rzn &\to \aff(\rzn) \\
    \boldsymbol{\theta} &\mapsto (I, \boldsymbol{\theta})
  \end{split}
\end{equation*}
\begin{equation*}
  \begin{split}
    \mathrm{Lin}: \aff(\rzn) &\to \glnz \\
    (A,\boldsymbol{\theta}) &\mapsto A,
  \end{split}
\end{equation*}
\noindent
and $I$ denotes the identity in $\glnz$. The sequence of equation
(\ref{eq:16}) admits a splitting
\begin{equation*}
  \begin{split}
    \sigma: \glnz &\to \aff(\rzn) \\
    A &\mapsto (A,\mathbf{0}),
  \end{split}
\end{equation*}
\noindent
where $\mathbf{0}$ denotes the identity in the group $\rzn$. The injections
$\sigma, \tau$ defined above give rise to interesting bundles via the
following well-known theorem, stated below without proof
(cf. Section 5 in \cite{huse}).

\begin{thm}\label{thm:subgroup}
  Let $G, H$ be topological groups and let
  $\iota: H \hookrightarrow G$ be a monomorphism. There exists a
  bundle
  $$ G/H \hookrightarrow \B H \to \B G, $$
  \noindent
  where $\B G, \B H$ are the classifying spaces for $G$ and $H$
  respectively. If, in addition, $\iota(H) \lhd G$,
  then the above bundle is a $G/H$-principal bundle.
\end{thm}

Applying Theorem \ref{thm:subgroup} to the inclusion $\tau : \rzn
\to \aff(\rzn)$ and to the splitting $\sigma : \glnz \to \aff(\rzn)$
defined above, obtain two bundles
\begin{subequations}\label{eq:25}
  \begin{align}
    & \univcover \label{eq:26} \\
    & \univlf \label{eq:11}
  \end{align}
\end{subequations}

\begin{rk}\label{rk:group_structure_on_fibres}
  Note that since $\sigma(\glnz)$ is not a normal subgroup of
  $\aff(\rzn)$, the fibres of the bundle of equation (\ref{eq:11}) are
  not naturally endowed with the structure of a group.
\end{rk}

\begin{defn}[Universal Lagrangian bundles
  \cite{sepe_topc}] \label{defn:top_univ_lf}
  For each $n$, the bundle of equation (\ref{eq:11}) is called the
  \emph{universal Lagrangian bundle} of dimension $n$.
\end{defn}

\begin{notation}
  Throughout this paper, $n$ is fixed, unless otherwise stated, and the
  universal Lagrangian bundle of dimension $n$ is
  referred to simply as the universal Lagrangian
  bundle.
\end{notation}

\begin{rk}\label{rk:affine_structure}
  The fibres of a universal Lagrangian bundle are
  endowed with an affine structure which makes them affinely
  diffeomorphic to $\rzn$. This is because the subgroup $\tau(\rzn)$
  of $\aff(\rzn)$ can be identified with the affine manifold $\rzn$
  (cf. Remark \ref{rk:affine_structure_translations}).
\end{rk}

The following lemma states what the homotopy groups of $\baff$ are;
its proof is omitted as it can be found in \cite{sepe_topc}. 

\begin{lemma}[Homotopy groups of $\baff$ \cite{sepe_topc}] \label{lemma:htygps}
  The homotopy groups of $\baff$ are given by
  \begin{equation*}
    \pi_i(\baff) =
    \begin{cases}
      \glnz & \text{if $i=1$,} \\
      \Z^n & \text{if $i=2$,} \\
      0 & \text{otherwise.}
    \end{cases}
  \end{equation*}
\end{lemma}

\subsubsection*{Universality}
Fix a Lagrangian bundle $\lf$. In light of Theorem
\ref{thm:existence_of_plb}, it is classified topologically by the
homotopy class of a map
$$ \chi: B \to \baff, $$
\noindent
which is the classifying map of the associated principal
$\aff(\rzn)$-bundle. The following theorem shows that the bundles of
Definition \ref{defn:top_univ_lf} are universal in the sense that
$\lf$ is isomorphic to the pull-back of the universal Lagrangian
bundle along $\chi$. The result below is presented without proof as it
follows from a simple observation for general semidirect products.

\begin{thm}[Universality] \label{thm:universality}
  Let $\lf$ be a Lagrangian bundle and let $\chi : B \to
  \baff$ denote its classifying map. Then $\lf$ is
  isomorphic to the pull-back of the universal Lagrangian
  bundle along $\chi$
  \begin{equation*}
    \xymatrix{ \chi^* \mathrm{B}\glnz \ar[r]^-{\Xi} \ar[d] &
      \mathrm{B}\glnz \ar[d] \\
      B \ar[r]_-{\chi} & \baff.}
  \end{equation*}
\end{thm}
 
\begin{rk}\label{rk:no_gp_structure_fibres_lf}
  As observed in the literature (\textit{e.g.}
  \cite{bates_snia,dui,luk}), the fibres of a Lagrangian bundle
  are not naturally equipped with a group structure. This can be
  proved directly using Theorem \ref{thm:universality} and
  Remark \ref{rk:group_structure_on_fibres}.
\end{rk}

\subsection{Topological invariants}\label{sec:topol-invar-lagr} 
In light of Theorem \ref{thm:universality}, the topological invariants
of Lagrangian bundles are pull-backs of \emph{universal} invariants,
which characterise the universal Lagrangian bundle. 

\subsubsection*{Monodromy} \label{sec:monodromy}
Fix a Lagrangian bundle $\lf$ with classifying map
$\chi: B \to \baff$.

\begin{defn}[Monodromy \cite{sepe_topc}]\label{defn:monodromy}
  Let $b_0 \in B$ be a basepoint. The \emph{monodromy} of $\lf$ is
  defined to be the homomorphism
  $$\chi_{*} : \pi_1(B;b_0) \to \pi_1(\baff; \chi(b_0)) \cong \glnz. $$
\end{defn}

\begin{rk}\label{rk:monodromy}
  \mbox{}
  \begin{enumerate}[label=\roman*),ref=\roman*]
  \item \label{item:21} There is a notion of \emph{universal}
    monodromy, which arises from considering the classifying map 
    of the universal Lagrangian bundle 
    $$\ulf.$$
    \noindent
    This map is (up to homotopy) just the identity $\mathrm{id}: \baff \to
    \baff$;
  \item \label{item:22} The choice of base point $b_0 \in B$ may affect
    the image of the homomorphism $\chi_*$, without, however, changing
    its conjugacy class in $\glnz$. The \emph{free monodromy} of a
    Lagrangian bundle $\lf$ is defined to be the conjugacy class
    $[\chi_*]$ of the image
    of the monodromy and is therefore independent of the
    choice of base point $b \in B$;
  \item \label{item:35} The bundle
    $$ \univcover $$
    \noindent
    of equation (\ref{eq:26}) is the universal
    covering of $\mathrm{B}\aff(\rzn)$. Thus if the (free) monodromy
    $\chi_*$ of a Lagrangian bundle $\lf$ is trivial, the
    classifying map $\chi$ admits a lift $\tilde{\chi} : B \to
    \mathrm{B}(\rzn)$. This implies that the structure group of the
    Lagrangian bundle can be reduced to the group $\rzn$
    (cf. Theorem 5.1 in \cite{huse}).
  \end{enumerate}
\end{rk}

The monodromy of a Lagrangian bundle $\lf$ was originally defined to be the
homotopy class of the classifying map of the associated period lattice bundle
$\plb$ (cf. \cite{dui}). This can be seen using universal
Lagrangian bundles. 
 
\begin{defn}[Universal period lattice bundles]\label{defn:uplb}
  The \emph{universal period lattice bundle} associated to the
  universal Lagrangian bundle 
  $$\ulf$$ 
  \noindent
  is the induced system of local coefficients with fibre
  $\Hh_1(\rzn;\Z)$, denoted by 
  $$\univplb.$$
\end{defn}

The pull-back $\chi^* P_n \to B$ is isomorphic to the system of local
coefficients obtained by replacing each fibre $\rzn$ of $\chi^* \bgl \to B$
with $\Hh_1(\rzn;\Z)$. By Theorem \ref{thm:universality}, $\chi^* \bgl
\to B$ is isomorphic to the Lagrangian bundle $\lf$ whose
classifying map is given by $\chi$. The following lemma, stated below
without proof, shows that $\chi^* P_n \to B$ is isomorphic to the
period lattice bundle $P \to B$ associated to $\lf$. Its proof follows
from ideas in \cite{dui}.

\begin{lemma}\label{lemma:pullbac_plb_gives_plb}
  Let $P \to B$ denote the period lattice bundle associated to the
  Lagrangian bundle $\lf$ (with classifying map $\chi$) as in
  Definition \ref{defn:plb}. Then 
  $\chi^* P_n \to B$ is isomorphic to $P \to B$.
\end{lemma}

The (free) monodromy
$\chi_*$ of the Lagrangian bundle $\lf$ determines the homotopy
class of the classifying map of $\chi^* P_n \to B$, which is
isomorphic to the period lattice bundle $P \to B$ by Lemma
\ref{lemma:pullbac_plb_gives_plb}. This proves that the monodromy of
Definition \ref{defn:monodromy}
defined above is equivalent to the notion given in \cite{dui}.

\subsubsection*{Chern class}\label{sec:chern-class}
Free monodromy is not the only topological
invariant of a Lagrangian bundle. This can be seen by considering
the case when the (free) monodromy vanishes and the topological
classification reduces to that of principal
$\rzn$-bundles. In this case, there is only one other topological
invariant associated to the given fibre bundle, namely the obstruction to
the existence of a section. In what follows, the corresponding
obstruction for any Lagrangian bundle is defined using universal
Lagrangian bundles.\\

The approach taken below is obstruction theoretic, which is suitable
since all spaces involved are homotopic to CW complexes. The idea is
to find the obstruction to the 
existence of a section of the universal Lagrangian
bundle. A section of this bundle is a lift of the
identity map $\mathrm{id}:\baff \to \baff$, \textit{i.e.} a map
(denoted by the dotted arrow) which makes the following diagram
commute
$$\xymatrix{ & \bgl \ar[d]^-{\sigma} \\
  \baff \ar[r]^-{\text{id}} \ar@{.>}[-1,1] & \baff} $$
\noindent
As the total space and the fibres are path-connected, a section can be
defined on the 1-skeleton of $\baff$. The problem of extending this
section to the 2-skeleton can be tackled cell by cell; however, there
needs to be extra care as $\baff$ is not simply connected. Set $\varpi
 = \pi_1(\baff)$, fix a CW decomposition of $\baff$ and
a $\varpi$-equivariant CW decomposition of its universal cover
$\widetilde{\baff}$. Using standard
techniques in obstruction theory (cf. Theorem 7.37 in \cite{davis_kirk}, Theorem 5.5 in \cite{white}), it can
be shown that the obstruction to the existence of an extension to the
2-skeleton is a cocycle in
$$ \mathrm{Hom}_{\Z[\varpi]}(\mathrm{C}_2(\widetilde{\baff});\Z^n), $$
\noindent
where $\mathrm{C}_2(\widetilde{\baff})$ and $\Z^n$ are
naturally $\Z[\varpi]$-modules: the former because of the choice of
$\varpi$-equivariant CW decomposition of $\widetilde{\baff}$ and the
latter via the topological monodromy representation of the topological
universal Lagrangian bundle.

\begin{defn}[Universal Chern class] \label{defn:univ_cc}
  The cohomology class of the above cocycle 
  $$c_U \in \mathrm{H}^2(\baff;\Z^n_{\text{id}_*}), $$
  \noindent
  is called the \emph{universal Chern class}. 
\end{defn}

The importance of the universal Chern class $c_U$ is highlighted in
Section \ref{sec:spectral-sequence} where it is used to study the
Leray-Serre spectral sequence associated to the universal
Lagrangian bundle.

\begin{defn} \label{defn:cc}
  Let $\lf$ be a Lagrangian bundle with classifying map $\chi: B
  \to \baff$. Its \emph{Chern class} is defined to be the pull-back
  $$ \chi^* c_U \in \mathrm{H}^2(B;\Z^n_{\chi_*}).$$
\end{defn}

\begin{rk}\label{rk:chern_class}
  \mbox{}
  \begin{enumerate}[label=\roman*)]
  \item \label{item:17} \textit{A priori}, the cohomology class
    $\chi^* c_U$
    is only the \emph{primary} obstruction to the existence of a
    section for $\lf$, \textit{i.e.} the vanishing of $\chi^* c_U$ may
    not mean that there exists a section $s : B \to M$. However, if
    $\chi^* c_U =0$, there exists a section
    defined on the 2-skeleton of $B$; it then follows from standard
    arguments in obstruction theory that this section can be extended to 
    $B$. Thus $\chi^*c_U$ is the only obstruction
    to the existence of a section; 
  \item \label{item:13} In \cite{sepe_topc} the Chern class of a
    Lagrangian bundle is defined to be the obstruction to the
    existence of a lift
    $$\xymatrix{ & \bgl \ar[d] \\
      B \ar[r]^-{\chi} \ar@{.>}[-1,1] & \baff.} $$
    \noindent
    This obstruction coincides with the obstruction to the existence
    of a section $s: B \to \chi^* \bgl$. Since $\lf$ is isomorphic to
    the pull-back of the universal Lagrangian bundle along $\chi$, the
    latter is just $\chi^* c_U$ and the above definition coincides with that of
    \cite{sepe_topc}.
  \end{enumerate}
\end{rk}

\subsubsection*{Sharpness}\label{sec:sharpness}
The two topological invariants defined above completely determine the
topological type of a Lagrangian bundle, as mentioned in
\cite{bates_ob,daz_delz,dui,luk,vu_ngoc,zun_symp2}. This can
also be seen using arguments in equivariant obstruction
theory which are akin to those used to define the universal Chern
class. Thus the following sharpness theorem is only stated.

\begin{thm}[Sharpness \cite{daz_delz,dui,luk,sepe_topc}] \label{thm:sharpness}
  Two Lagrangian bundles 
  $$\lf \quad , \quad \rzn \hookrightarrow (M',\omega')\to B$$
  \noindent
  are isomorphic (as fibre bundles) if and only if their free
  monodromies and Chern classes coincide.
\end{thm}


\section{Construction Problem}\label{sec:construction-problem}
\subsection{Relation between integral affine geometry and Lagrangian
  bundles}\label{sec:pecul-lagr-bundl}
he classification theory developed in Section
\ref{sec:defin-basic-prop} does not suffice to tackle the problem of
constructing Lagrangian bundles over a fixed base space $B$. This is
because the above classification is merely topological and not
symplectic. In fact, the (free) monodromy and the Chern class are not unique to
Lagrangian bundles, as illustrated below.

\begin{defn}[Affine $\rzn$-bundles \cite{baier}]\label{defn:affine_bundles}
  A locally trivial $\tn$-bundle is called \emph{affine} if its
  structure group reduces to $\aff (\rzn)$.
\end{defn}

\begin{rk}\label{rk:affine_bundles}
  The fibres of an affine $\rzn$-bundle can be smoothly identified with
  the affine manifold $\rzn$. The topological
  classification of affine $\rzn$-bundles can be carried out as in Section
  \ref{sec:defin-basic-prop} above, since their structure group reduces to
  $\aff(\rzn)$; in particular, Theorem
  \ref{thm:sharpness} applies to this more general class.
\end{rk}

Not all locally trivial $\tn$-bundles are affine $\rzn$-bundles as illustrated
in an example due to Baier in \cite{baier}. Moreover, it can be shown
that not all affine $\rzn$-bundles are Lagrangian by considering the
case of principal $\rzn$-bundles over $S^n$
(cf. \cite{sepe_topc}). What distinguishes affine $\rzn$ from
Lagrangian bundles is the existence of a suitable symplectic form on
the total space; this object imposes severe restrictions on the
topology of the base space, as illustrated in the next section.

\subsubsection*{Action-angle coordinates and integral affine
  geometry} \label{sec:canon-coord} 
One of the crucial consequences of the Liouville-Mineur-Arnol'd
theorem is the existence of canonical coordinates in a neighbourhood
of a fibre of a Lagrangian bundle. These are known as
\emph{action-angle} coordinates and play a very important role in
Hamiltonian mechanics. The following theorem illustrates the existence
of such coordinates and investigates how they change; it is stated
below without proof as it is well-known.

\begin{thm}[Existence of local action-angle coordinates,
  Section 49 in \cite{arnold}, \cite{dui}]\label{thm:ex_local_action_angle} 
  Let $\flf$ be a Lagrangian bundle. There exists an open cover
  $\mathcal{U} = \{U_{\alpha}\}$ of $B$ and coordinate charts
  $\phi_{\alpha}: U_{\alpha} \to \R^n$ inducing symplectic
  trivialisations
  $$ \Upsilon_{\alpha} : (\cotan U_{\alpha}/P_{\alpha},
  \omega_{0,\alpha}) \to (\pi^{-1}(U_{\alpha}), \omega_{\alpha}), $$
  \noindent
  where $P_{\alpha} \subset (\cotan U_{\alpha}, \Omega_{\alpha})$ is
  the Lagrangian submanifold
  $$ P_{\alpha} = \{(\mathbf{a}_{\alpha},\mathbf{p}_{\alpha}) \in
  \cotan U_{\alpha} :\, \mathbf{p}_{\alpha} \in \Z \langle \de
  a^1_{\alpha}, \ldots, \de a^n_{\alpha} \rangle \}, $$
  \noindent
  $a^1_{\alpha},\ldots,a^n_{\alpha}$ are local coordinates on
  $U_{\alpha}$ induced by $\phi_{\alpha}$, and
  $(\mathbf{a}_{\alpha},\mathbf{p}_{\alpha})$ are canonical
  coordinates on $\cotan U_{\alpha}$. The local coordinates 
  $(\mathbf{a}_{\alpha},\boldsymbol{\theta}_{\alpha})$ induced by
  $\Upsilon_{\alpha}$ are called \emph{action-angle} coordinates and
  satisfy
  $$ \Upsilon_{\alpha}^* \omega_{\alpha} = \omega_{0,\alpha} = \sum\limits_{i=1}^n \de
  a^i_{\alpha} \wedge \de \theta^i_{\alpha} . $$
\end{thm}

For the purposes
of this work, the most important consequence of Theorem
\ref{thm:ex_local_action_angle} is that the base space $B$ of a Lagrangian
bundle inherits the structure of an \emph{integral affine
  manifold}. 

\begin{defn}[Integral affine manifolds]\label{defn:iam}
  An $n$-dimensional affine manifold $\am$ is said to be \emph{integral} if the
  coordinate changes in its affine structure $\mathcal{A}$ are
  \emph{integral affine} maps of $\R^n$, \textit{i.e.} elements of
  the group
  $$ \affz (\R^n) := \glnz \ltimes \R^n.$$
\end{defn}

The following lemma shows that the coordinate neighbourhoods
$\phi_{\alpha}$ of Theorem \ref{thm:ex_local_action_angle} give rise
to an integral affine structure on $B$ and that, conversely, all
integral affine manifolds arise as the base space of some Lagrangian
bundle (for a proof, see \cite{sepe_topc}). 

\begin{lemma}[\cite{dui,luk,sepe_topc}]\label{lemma:base_space_iam}
  A manifold $B$ is the base space of a Lagrangian bundle if and
  only if it is an integral affine manifold.
\end{lemma}

\begin{rk}\label{rk:ex_symp_ref_bundle}
  Associated to an integral affine manifold $\am$
  is a Lagrangian submanifold $P_{\am} \subset 
  (\cotan B, \Omega)$, locally given by
  $$ P_{\alpha} := \{ (\mathbf{a}_{\alpha}, \mathbf{p}_{\alpha}) \in
  \cotan U_{\alpha} :\, \mathbf{p}_{\alpha} \in \Z \langle \de
  a^1_{\alpha},\ldots, \de a^n_{\alpha}\rangle \}, $$
  \noindent
  where $U_{\alpha} \subset B$ is a coordinate neighbourhood with
  integral affine coordinates $\mathbf{a}_{\alpha}$, and
  $\mathbf{p}_{\alpha}$ denote canonical coordinates on the fibres of
  $\cotan U_{\alpha}$. The canonical symplectic form $\Omega$ on
  $\cotan B$ descends
  to a symplectic form $\omega_0$ on the quotient $\cotan B/P_{\am}$;
  thus,
  \begin{equation*}
    (\cotan B /P, \omega_0) \to B. 
  \end{equation*}
  \noindent
  is a Lagrangian bundle with a globally defined Lagrangian section,
  namely the zero section. 
\end{rk}

Note that the construction of Remark \ref{rk:ex_symp_ref_bundle} only depends on
the integral affine structure on $B$ and, thus, the following
are well-defined.

\begin{defn}[Topological reference bundle
  \cite{sepe_klein}]\label{defn:top_ref_bundle}
  Let $\am$ be an integral affine manifold. The isomorphism class of
  the bundle
  $$\rlf$$
  \noindent
  as an affine $\rzn$-bundle is
  called the \emph{topological reference} bundle for $\am$.
\end{defn}

\begin{defn}[Symplectic reference bundle
  \cite{sepe_topc,zun_symp2}]\label{defn:symp_reference_bundle}
  The bundle 
  $$\rlf$$ 
  \noindent
  is the \emph{symplectic reference}
  bundle for the integral affine manifold $\am$.
\end{defn}

\begin{rk}[Symplectic reference bundle]\label{rk:up_to_symplecto}
  It is important to notice that the symplectic reference
  bundle is equipped with a distinguished Lagrangian section. Any
  other Lagrangian bundle which admits a Lagrangian section is
  fibrewise symplectomorphic to the symplectic reference bundle,
  \textit{i.e.} there exists a fibre bundle isomorphism which is also a
  symplectomorphism of the total spaces. Such a symplectomorphism is
  sometimes referred to as a \emph{polarisation} in the literature,
  \textit{e.g.} \cite{misha}.
\end{rk}

The submanifold $P_{\am} \subset \cotan B$ constructed above plays an
important role in building Lagrangian bundles. 

\begin{defn}[Period lattice bundle of an integral affine
  manifold]\label{defn:integral_affine_periods} 
  Let $\am$ be an integral affine manifold. The Lagrangian submanifold
  $P_{\am} \subset (\cotan B, \Omega)$ constructed in Remark
  \ref{rk:ex_symp_ref_bundle} is called the \emph{period lattice
    bundle} associated to $\am$.
\end{defn}
 
Let $\lf$ be a Lagrangian bundle. Theorem \ref{thm:existence_of_plb}
and Lemma \ref{lemma:base_space_iam} show that $B$ can be covered by
integral affine coordinate neighbourhoods $U_{\alpha}$ over which
the bundle is trivial. Let
$(\mathbf{a}_{\alpha},\boldsymbol{\theta}_{\alpha})$ be action-angle 
coordinates on $\pi^{-1}(U_{\alpha})$, so that the restriction of
$\omega$ to $\pi^{-1}(U_{\alpha})$, denoted by $\omega_{\alpha}$, is
$$ \sum\limits_{i=1}^n \de a^i_{\alpha} \wedge \de
\theta^i_{\alpha}.$$
\noindent
The transition functions $\varphi_{\beta \alpha}$ with respect to the
above choice of local trivialisations are given by (cf. \cite{zun_symp2})
\begin{equation}
  \label{eq:41}
  \varphi_{\beta
    \alpha}(\mathbf{a}_{\alpha},\boldsymbol{\theta}_{\alpha}) =
  (A_{\beta \alpha}
  \mathbf{a}_{\alpha} + \mathbf{d}_{\beta \alpha}, A^{-T}_{\beta \alpha}
  \boldsymbol{\theta}_{\alpha} + \mathbf{g}_{\beta
    \alpha}(\mathbf{a}_{\alpha})),
\end{equation}
\noindent
where $(A_{\beta \alpha}, \mathbf{d}_{\beta \alpha}) \in \affz (\R^n)$
is a change of integral affine coordinates, and
$\mathbf{g}_{\beta \alpha} : U_{\alpha} \cap U_{\beta} \to 
\rzn$ is a smooth map which is constrained by
\begin{equation*}
  \varphi^*_{\beta \alpha} \omega_{\beta} = \omega_{\alpha},
\end{equation*}
\noindent
\textit{i.e.} the local symplectic forms patch together to
yield $\omega$ on $M$.

\subsubsection*{Monodromy and integral affine geometry}
Let $\lf$ be a Lagrangian bundle. The induced
integral affine structure $\mathcal{A}$ on the base space can be
understood in terms of a flat, torsion-free connection $\nabla$ on $\T
B$ whose holonomy takes values in $\affz (\R^n)$ (cf. \cite{aus}).

\begin{defn}[Affine and linear holonomy
  \cite{aus,goldman}]\label{defn:aff_lin_holo} 
  The holonomy of $\nabla$, denoted by 
  $$\mathfrak{a} : \pi_1(B;b) \to \affz(\R^n),$$
  \noindent
  is called the \emph{affine holonomy} of the integral affine manifold
  $\am$. Composing $\mathfrak{a}$ with the projection
  $$ \mathrm{Lin} :\affz(\R^n) \to \glnz, $$
  \noindent
  obtain the \emph{linear holonomy} $\mathfrak{l} : \pi_1(B;b) \to \glnz$.
\end{defn}

\begin{rk}\label{rk:choice_of_developing_map}
  The affine holonomy is only well-defined up to an explicit choice of
  affine structure and up to a choice of basepoint. Throughout
  this work, whenever an affine/linear holonomy homomorphism is
  considered, it is understood that the basepoint and the explicit
  integral affine structure are fixed.
\end{rk}

The linear monodromy $\mathfrak{l}$ of $\am$ is related to the
monodromy $\chi_*$ of the Lagrangian bundle $\lf$ via
\begin{equation}
  \label{eq:45}
  \chi_* = \mathfrak{l}^{-T},
\end{equation}
\noindent
where $-T$ denotes inverse transposed (cf. equation (\ref{eq:41})). However, the
linear holonomy alone does not determine the symplectic reference bundle up to
fibrewise symplectomorphism, as illustrated by the next simple
example.

\begin{ex}[Some symplectic reference bundles
  \cite{luk,sepe_klein}]\label{ex:symp_ref_bundles}
  Let $\R/\Z, \R/2\Z$ be the integral affine
  manifolds obtained by taking the quotient of $\R$ by the lattices
  $\Z$ and $2\Z$, which are affinely isomorphic, but not integrally
  affinely isomorphic. Their linear
  holonomies are trivial, so that their symplectic reference
  Lagrangian bundles are principal $S^1$-bundles with a section,
  \textit{i.e.} they are globally trivial. Thus their topological
  reference bundles are isomorphic. However, the total spaces of
  their symplectic reference bundles are symplectomorphic
  to
  $$ (\R^2/\Z^2, \omega_1),\,(\R^2/2\Z \oplus \Z, \omega_2) $$
  respectively, where $\omega_1, \omega_2$ are symplectic forms which
  descend from the standard symplectic form $\Omega$ obtained by
  considering $\R^2 \cong \cotan \R$ (cf. \cite{luk}). These two
  symplectic manifolds cannot be symplectomorphic since the total
  spaces have different volumes. This example can been refined to work
  in the case when the total spaces have the same volume (cf. \cite{sepe_klein}). 
\end{ex}

The idea that lies at the heart of Example \ref{ex:symp_ref_bundles}
is that the symplectic reference bundle $\rlf$ associated
to an integral affine manifold $\am$ contains information about the
integral affine structure $\mathcal{A}$. This idea is further explored
in Section \ref{sec:relat-integr-affine}.

\subsection{Almost Lagrangian bundles} \label{sec:almost-lagr-bundl}
In light of Section \ref{sec:pecul-lagr-bundl}, it is
now possible to state main question that this paper tackles.

\begin{qn}\label{qn:almost_lag}
  Let $\am$ be an $n$-dimensional integral affine manifold with linear
  holonomy $\mathfrak{l}$, and let $\alf$ be an
  affine $\rzn$-bundle classified by the homotopy class of a map
  $$ \chi: B \to \baff,$$
  \noindent
  which satisfies the condition of equation (\ref{eq:45}). What is the
  obstruction to endowing $M$ with a symplectic form $\omega$ which
  makes the bundle Lagrangian?
\end{qn}

\begin{defn}[Almost Lagrangian bundles]\label{defn:almost_lag_bundles}
  For a fixed $n$-dimensional integral affine manifold $\am$ with
  linear holonomy $\mathfrak{l}$, the
  affine $\rzn$-bundles over $B$ whose monodromy satisfies the
  condition of equation (\ref{eq:45}) are called \emph{almost Lagrangian}.
\end{defn}

\begin{rk}[Classification of almost Lagrangian
  bundles]\label{rk:classification_almost_Lagrangian} 
  The isomorphism classes of almost Lagrangian bundles over an
  $n$-dimensional integral affine manifold $\am$ 
  with linear holonomy $\mathfrak{l}$ are in $1-1$ correspondence with
  elements of $\Hh^2(B;\Z^n_{\mathfrak{l}^{-T}})$ (cf. Remark
  \ref{rk:affine_bundles}). The almost Lagrangian
  bundle corresponding to $0 \in \Hh^2(B;\Z^n_{\mathfrak{l}^{-T}})$
  is necessarily Lagrangian and the symplectic form on its total space
  can always be chosen so as to admit a Lagrangian section (cf. Remark
  \ref{rk:ex_symp_ref_bundle}). 
\end{rk}

\begin{rk}[Local trivialisations of almost Lagrangian
  bundles] \label{rk:local_trivialisations_of_alb}
  Fix an almost Lagrangian bundle $\alf$ over $\am$ and denote by $P$
  its period lattice bundle
  \begin{equation*}
    \Hh_1(\rzn;\Z) \hookrightarrow P \to B,
  \end{equation*}
  \noindent
  whose isomorphism class is determined by a homomorphism
  $$ \pi_1(B) \to \mathrm{Aut}(\Z^n) \cong \glnz, $$
  \noindent
  which, up to a choice of basepoint, equals the inverse transposed
  $\mathfrak{l}^{-T}$ of the linear holonomy of $\am$ by
  definition. If $P_{\am} \to B$ denotes the period lattice bundle
  associated to $\am$ (cf. Definition
  \ref{defn:integral_affine_periods}), then there is an isomorphism
  \begin{equation}
    \label{eq:77}
    P \cong P_{\am}.
  \end{equation}
  \noindent
  Let $\mathcal{U} =\{U_{\alpha}\}$ be a good (in the sense of Leray)
  open cover of $B$ by 
  integral affine coordinate neighbourhoods. The restriction 
  \begin{equation}
    \label{eq:78}
    \rzn \hookrightarrow \pi^{-1}(U_{\alpha}) \to U_{\alpha}
  \end{equation}
  \noindent
  is an almost Lagrangian bundle. Since $U_{\alpha}$ is contractible,
  the above bundle is trivial and therefore there exists a section
  $s_{\alpha} : U_{\alpha} \to \pi^{-1}(U_{\alpha})$. Fix such a
  section. Since 
  $P_{\alpha} = P|_{U_{\alpha}} \to U_{\alpha}$ is trivial, the above
  identification can be extended to a trivialisation 
  $$ \pi^{-1} (U_{\alpha}) \cong (P_{\alpha} \otimes_{\Z} \R)
  /P_{\alpha} $$
  \noindent
  using the section $s_{\alpha}$. The isomorphism of equation
  (\ref{eq:77}) extends to yield an isomorphism
  $$ P \otimes_{\Z} \R \cong P_{\am} \otimes_{\Z} \R \cong \cotan
  B, $$
  \noindent
  where the second isomorphism follows from the definition of
  $P_{\am}$. In particular, the restriction of this isomorphism to
  $\pi^{-1}(U_{\alpha})$ defines a trivialisation 
  \begin{equation*}
    \Upsilon_{\alpha} : \pi^{-1}(U_{\alpha}) \to \cotan
    U_{\alpha}/P_{\am}|_{U_{\alpha}}
  \end{equation*}
  \noindent
  which maps the section $s_{\alpha}$ to the zero section of $\cotan
  U_{\alpha}/P_{\am}|_{U_{\alpha}} \to U_{\alpha}$. Note that if $\omega_0$ denotes
  the symplectic form making the bundle 
  $$ \cotan B/P_{\am} \to B$$
  \noindent
  into the symplectic reference bundle associated to $\am$,
  and $\omega_{0, \alpha}$ is its restriction to $\cotan
  U_{\alpha}/P_{\am}|_{U_{\alpha}}$, then
  $$\Upsilon_{\alpha}^*\omega_{0,\alpha}$$
  \noindent
  makes the bundle of equation (\ref{eq:78}) Lagrangian. Furthermore,
  the section $s_{\alpha}$ is also Lagrangian. \\
  
  If $(\mathbf{a}_{\alpha},\boldsymbol{\theta}_{\alpha})$ denote action-angle
  coordinates on $\cotan U_{\alpha}/P_{\am}|_{U_{\alpha}}$, the
  transition functions
  $$ \varphi_{\beta \alpha}=\Upsilon_{\beta} \circ \Upsilon^{-1}_{\alpha} : \cotan
  (U_{\alpha} \cap U_{\beta})/P_{\am}|_{U_{\alpha}\cap U_{\beta}} \to \cotan
  (U_{\alpha} \cap U_{\beta})/P_{\am}|_{U_{\alpha}\cap U_{\beta}} $$
  \noindent
  for the above choice of trivialisations take the familiar form
  (cf. equation (\ref{eq:41}))
  \begin{equation}
    \label{eq:80}
    \varphi_{\beta\alpha}(\mathbf{a}_{\alpha},\boldsymbol{\theta}_{\alpha}) =
    (A_{\beta \alpha} \mathbf{a}_{\alpha} + \mathbf{d}_{\beta \alpha},
    A^{-T}_{\beta \alpha} 
    \boldsymbol{\theta}_{\alpha} + \mathbf{g}_{\beta
      \alpha}(\mathbf{a}_{\alpha})),    
  \end{equation}
  \noindent
  where the first component comes from an affine change of coordinates
  on $U_{\alpha} \cap U_{\beta}$, and the linear part of the second
  component is determined by the transition functions for $P_{\am} \to
  B$ by construction. Note that the maps $\varphi_{\beta \alpha}$ are
  not necessarily fibrewise symplectomorphisms as the locally defined
  forms $\Upsilon^*_{\alpha} \omega_{0,\alpha}$ do not necessarily
  patch together.
\end{rk}

\begin{rk}[Equivalent definition of almost Lagrangian bundles
  \cite{daz_delz}]\label{rk:equivalent_definition} 
  Fix an $n$-dimensional integral affine manifold $\am$ with linear
  holonomy $\mathfrak{l}$, and let $P_{\am}
  \subset \cotan B$ denote its associated period lattice
  bundle. Denote by $\mathcal{P}$ the sheaf 
  of smooth sections of the projection $P_{\am} \to B$. Note that $P_{\am}$ is
  isomorphic to the period lattice bundle of any Lagrangian bundle
  over $B$ whose monodromy equals $\mathfrak{l}^{-T}$, as proved in
  \cite{dui}. Thus there is an isomorphism of cohomology groups
  $$ \Hh^i(B;\mathcal{P}) \cong \Hh^i(B;\Z^n_{\mathfrak{l}^{-T}}) $$
  \noindent
  for all $i$. The sheaf $\mathcal{P}$ fits into a
  short exact sequence (cf. \cite{dui})
  \begin{equation}
    \label{eq:50}
    0 \to \mathcal{P} \to \mathcal{C}^{\infty}(\cotan B) \to
    \mathcal{C}^{\infty} (\cotan B/P_{\am}) \to 0,
  \end{equation}
  \noindent
  where $\mathcal{C}^{\infty}(\cotan B)$ and
  $\mathcal{C}^{\infty}(\cotan B/P_{\am})$ are the sheaves of smooth
  sections of the cotangent bundle $\cotan B \to B$ and of the
  topological reference bundle $\cotan B/P_{\am} \to B$
  respectively. The long exact sequence in cohomology induced by
  equation (\ref{eq:50}) collapses to isomorphisms
  $$ \Hh^i(B;\mathcal{C}^{\infty}(\cotan B/P_{\am})) \cong
  \Hh^{i+1}(B;\mathcal{P}) $$
  \noindent
  for $i \geq 1$, since $\mathcal{C}^{\infty}(\cotan B)$ is a fine
  sheaf. In particular,
  $$ \Hh^1(B;\mathcal{C}^{\infty}(\cotan B/P_{\am})) \cong
  \Hh^2(B;\mathcal{P}). $$
  \noindent
  Hence
  \begin{equation}
    \label{eq:140}
    \Hh^1(B;\mathcal{C}^{\infty}(\cotan B/P_{\am})) \cong
    \Hh^2(B;\Z^n_{\mathfrak{l}^{-T}}).
  \end{equation}
  \noindent
  The group $\Hh^1(B;\mathcal{C}^{\infty}(\cotan B/P_{\am}))$ classifies the
  isomorphism classes of bundles over $B$ which are locally isomorphic
  to $\cotan B/P_{\am} \to B$ and have structure sheaf
  $\mathcal{C}^{\infty}(\cotan B/P_{\am})$
  (cf. \cite{daz_delz,gro}). In light of Remark
  \ref{rk:classification_almost_Lagrangian} and 
  equation (\ref{eq:140}), an almost Lagrangian bundle satisfies this
  condition and, conversely, any bundle over $B$ with the above
  properties is almost Lagrangian. This is the point of view
  taken in \cite{daz_delz}, and, more generally, in other works in the
  literature, \textit{e.g.} \cite{dui,luk,zun_symp2}.
\end{rk}

\begin{rk}[Relation to integrable systems
  \cite{fasso}]\label{rk:relation_to_integrable_systems}
  Almost Lagrangian bundles are the correct geometric setting for studying the
  type of generalised Liouville integrability that Fass\`o and
  Sansonetto consider in \cite{fasso}. It can be shown that the
  total space of an almost Lagrangian bundle admits an appropriate
  non-degenerate 2-form with respect to which the fibres are maximally
  isotropic submanifolds. This work will appear in \cite{sans_sepe}.
\end{rk}

Not all almost Lagrangian bundles are Lagrangian. Fix an
$n$-dimensional integral affine manifold $\am$
with linear holonomy $\mathfrak{l}$, and let $P_{\am} \subset \cotan B$ be
the period lattice bundle associated to $\am$ (cf. Definition
\ref{defn:integral_affine_periods}). The sheaf of smooth sections
$\mathcal{P}$ of $P_{\am}\to B$ fits in a short exact sequence
\begin{equation*}
  0 \to \mathcal{P} \to \mathcal{Z}(\cotan B) \to \mathcal{Z}(\cotan
  B/P_{\am}) \to 0,
\end{equation*}
\noindent
where $\mathcal{Z}(\cotan B)$ and  $\mathcal{Z}(\cotan B/P_{\am})$ are the
sheaves of closed sections of $\cotan B \to B$ and $\cotan B/P_{\am} \to B$
respectively. The induced long exact sequence in 
cohomology groups yields a homomorphism
$$ \mathcal{D}_{\am} : \Hh^2(B;\mathcal{P}) \to
\Hh^2(B;\mathcal{Z}(\cotan B)), $$
\noindent
where the subscript denotes the dependence upon the integral affine
structure of the manifold $B$. This homomorphism depends on
$\mathcal{A}$ as the latter determines $P_{\am}$ as a Lagrangian submanifold
of $(\cotan B, \Omega)$ (cf. Definition
\ref{defn:integral_affine_periods}). By the Poincar\'e Lemma
(cf. \cite{weinstein}),
$$ \Hh^2(B;\mathcal{Z}(\cotan B)) \cong \Hh^3(B;\R), $$
\noindent
where $\Hh^3(B;\R)$ denotes cohomology with real coefficients, and the
isomorphism between this cohomology theory and the cohomology with
coefficients in the constant sheaf $\mathcal{R}$ over $B$ with $\R$
coefficients is used tacitly. Hence there is a homomorphism
\begin{equation*}
   \mathcal{D}_{\am} : \Hh^2(B;\mathcal{P}) \to
   \Hh^3(B;\R).
\end{equation*}

\begin{thm}[Dazord and Delzant \cite{daz_delz}]\label{thm:daz_delz}
  Let $\alf$ be an almost Lagrangian bundle over $\am$ with linear holonomy
  $\mathfrak{l}$, and let 
  $$c \in \Hh^2(B;\Z^n_{\mathfrak{l}^{-T}}) \cong
  \Hh^2(B;\mathcal{P})$$
  \noindent
  be its Chern class. Then $\alf$ is Lagrangian if and only if  
  \begin{equation*}
    \mathcal{D}_{\am}c = 0.
  \end{equation*}
\end{thm}

The rest of this paper is devoted to proving that the
homomorphism $\mathcal{D}_{\am}$ is related to a differential on the
$\E_2$-page of the Leray-Serre spectral sequence of the topological
universal Lagrangian bundles constructed in Section
\ref{sec:defin-basic-prop}. This observation allows to give a
different proof of Theorem \ref{thm:daz_delz}, which highlights the
importance of integral affine geometry.

\section{A Spectral Sequence}\label{sec:spectral-sequence}
\subsection{A preparatory lemma} \label{sec:preparatory-lemma}
In this section, the universal Lagrangian bundle is
identified as the equivariant equivalent of the universal bundle
for principal $\rzn$-bundles. Free monodromy of an affine
$\rzn$-bundle is the obstruction for it to be a principal
$\rzn$-bundle. Therefore, if $\tau: \brzn \to \baff$ denotes the
bundle induced by the inclusion $\tau: \rzn \hookrightarrow \affz
(\R^n)$, then the following lemma holds.

\begin{lemma}\label{lemma:commutative_diagram}
  The bundle
  \begin{equation}
    \label{eq:3}
    \rzn \hookrightarrow \tau^* \bgl \to \brzn
  \end{equation}
  \noindent
  obtained by pulling back the universal Lagrangian bundle
  along the universal covering map $\tau: \brzn \to \baff$ is a
  universal bundle for principal $\rzn$-bundles.
\end{lemma}

\begin{rk}\label{rk:naturality_Chern_class}
  By functoriality of the Chern class, the obstruction to
  the existence of a section of the bundle of equation (\ref{eq:3})
  is given by
  $$\tau^* c_U \in \mathrm{H}^2(\brzn;\Z^n_{(\mathrm{id}\circ \tau)_*}),$$
  \noindent
  where $c_U$ denotes the universal Chern class. Note that
  $\pi_1(\brzn)$ is trivial, so that the 
  above coefficient system is just the constant $\Z^n$-system of
  coefficients on $\brzn$. Therefore
  \begin{equation}
    \label{eq:82}
    \tau^* c_U = c_{\rzn},
  \end{equation}
  \noindent
  where $c_{\rzn}$ is the obstruction to the existence of a section
  for the bundle 
  $$\rzn \hookrightarrow \mathrm{E} \rzn \to \mathrm{B}\rzn.$$
\end{rk}

\subsection{The spectral sequence of a universal Lagrangian
  bundle} \label{sec:spectr-sequ-univ} 
In this section the methods of \cite{charlap} are adapted to prove
that some differential on the $\mathrm{E}_2$-page of the Leray-Serre
spectral sequence of the universal Lagrangian bundle
is given (up to some isomorphisms) by taking cup products with the
universal Chern class
$$c_U \in
\mathrm{H}^2(\baff;\Z^n_{\text{id}_*}).$$
\noindent
This result is crucial to study the problem of
constructing Lagrangian bundles over a fixed
integral affine manifold.\\

Given a spectral sequence, the main idea in \cite{charlap} is to use
\emph{auxiliary} spectral sequences to reduce the problem of
determining differentials on the $\mathrm{E}_2$-page of the original
spectral sequence to a simpler one. While \cite{charlap} deals with
the cohomology of group extensions and, in particular, with abelian
extensions, the case of the universal Lagrangian bundle and, more
generally, of affine
$\rzn$-bundles can be thought of as a natural generalisation. This is
because the interesting part of the long exact sequence in homotopy
for an affine $\rzn$-bundle  $\alf$ is
$$ 0 \to \pi_2 M \to \pi_2 B \to \pi_1 \rzn \to \pi_1 M \to \pi_1 B
\to 1. $$

\subsubsection*{Auxiliary spectral
  sequences} \label{sec:auxil-spectr-sequ}
The $\mathrm{E}_2$-page of the Leray-Serre spectral
sequence for the universal Lagrangian bundle 
$$\ulf$$ 
\noindent
with $\Z$ coefficients is denoted by
$$ \mathrm{E}^{p,q}_2 \cong
\mathrm{H}^p(\baff;\mathrm{H}^q(\rzn;\Z)_{\rho_q}). $$
\noindent
The homomorphism
$$ \rho_q : \pi_1(\baff) \to \mathrm{Aut}(\mathrm{H}^q(\rzn;\Z)) $$
\noindent
classifies the local coefficient system defined by replacing each fibre $\rzn$
of the universal Lagrangian bundle with its $q$-th
cohomology group with integer coefficients 
$\mathrm{H}^q(\rzn;\Z)$. Henceforth, fix a basepoint in $\baff$, so
that the above homomorphisms are also fixed. Denote by 
$$ \breve{\rho}_1: \pi_1(\baff) \to
\mathrm{Aut}(\mathrm{H}_1(\rzn;\Z)) $$
\noindent
the homomorphism classifying the local coefficient system with fibre
$\Hh_1(\rzn;\Z)$ over $\baff$. Following \cite{charlap},
introduce \emph{auxiliary} spectral sequences, whose $\mathrm{E}_2$-pages are
given by
\begin{equation}
  \label{eq:83}
  \begin{split}
    \bar{\mathrm{E}}^{p,q}_2 &\cong
    \mathrm{H}^p(\baff;\mathrm{H}^q(\rzn;\mathrm{H}^1(\rzn;\Z))_{\bar{\rho}_q}), \\
    \hat{\mathrm{E}}^{p,q}_2 &\cong
    \mathrm{H}^p(\baff;\mathrm{H}^q(\rzn;\mathrm{H}_1(\rzn;\Z))_{\hat{\rho}_q}),
  \end{split}
\end{equation}
\noindent
where the above local coefficient systems are given by
\begin{equation*}
  \begin{split}
    \bar{\rho}_q = \rho_q \otimes \rho_1&: \pi_1(\baff) \to
    \mathrm{Aut}(\Hh^q(\rzn;\Z) \otimes \mathrm{H}^1(\rzn;\Z)), \\
    \hat{\rho}_q = \rho_q \otimes \breve{\rho}_1&:\pi_1(\baff) \to
    \mathrm{Aut}(\Hh^q(\rzn;\Z) \otimes \mathrm{H}_1(\rzn;\Z)), 
  \end{split}
\end{equation*}
\noindent
via the isomorphisms
\begin{equation*}
  \begin{split}
    \mathrm{H}^q(\rzn;\mathrm{H}^1(\rzn;\Z)) &\cong \Hh^q(\rzn;\Z)
    \otimes \mathrm{H}^1(\rzn;\Z), \\
    \mathrm{H}^q(\rzn;\mathrm{H}_1(\rzn;\Z)) &\cong \Hh^q(\rzn;\Z)
    \otimes \mathrm{H}_1(\rzn;\Z),
  \end{split}
\end{equation*}
\noindent
induced by the universal coefficient theorem. These spectral sequences
are henceforth referred to as the Leray-Serre spectral sequences with
$\Hh^1(\rzn;\Z)$ and $\Hh_1(\rzn;\Z)$ coefficients respectively.\\

There is a pairing
\begin{equation}
  \label{eq:84}
  \bar{\mathrm{E}}^{p,q}_2 \otimes_{\Z} \hat{\mathrm{E}}^{p',q'}_2
  \to \mathrm{E}_2^{p+p',q+q'}
\end{equation}
\noindent
induced by taking cup products and by the standard duality
$$ \mathrm{H}^1(\rzn;\Z) \otimes_{\Z} \mathrm{H}_1(\rzn;\Z) \to
\Z, $$
\noindent
called the \emph{auxiliary pairing} associated to the universal
Lagrangian bundle.

If $\de^{(2)},~
\bar{\de}^{(2)},~\hat{\de}^{(2)}$ denote the differentials on the
$\mathrm{E}_2$-page of the spectral sequences
$\mathrm{E}^{*,*}_2,~\bar{\mathrm{E}}^{*,*}_2,~\hat{\mathrm{E}}^{*,*}_2$
respectively, there is the multiplicative formula
\begin{equation}
  \label{eq:85}
  \de^{(2)}(x \cdot y) = \bar{\de}^{(2)}(x) \cdot y + (-1)^{p+q}x
  \cdot \hat{\de}^{(2)}(y),
\end{equation}
\noindent
where $x \in \hat{\mathrm{E}}_2^{p,q},~y \in
\bar{\mathrm{E}}_2^{p',q'}$ and $x \cdot y$ denotes the auxiliary
pairing between $x$ and $y$.\\

As shown in \cite{charlap}, there is an isomorphism
\begin{equation}
  \label{eq:86}
  \theta : \mathrm{E}^{p,1}_2 \to \bar{\mathrm{E}}^{p,0}_2
\end{equation}
\noindent
induced by
$$ \mathrm{H}^1(\rzn;\Z) \to \mathrm{H}^0(\rzn;
\mathrm{H}^1(\rzn;\Z)). $$
\noindent
The identity map
$$ \text{id} :  \mathrm{H}_1(\rzn;\Z) \to  \mathrm{H}_1(\rzn;\Z) $$
\noindent
defines an element $g^1 \in \mathrm{H}^1(\rzn;\mathrm{H}_1(\rzn;\Z))$,
since 
$$ \mathrm{H}^1(\rzn;\mathrm{H}_1(\rzn;\Z)) \cong
\mathrm{Hom}(\mathrm{H}_1(\rzn;\Z);\mathrm{H}_1(\rzn;\Z)).$$ 
\noindent
The element $g^1$, in turn, defines an element
$$ f^1 \in \hat{\mathrm{E}}^{0,1}_2 \cong \mathrm{H}^0(\baff;
\mathrm{H}^1(\rzn;\mathrm{H}_1(\rzn;\Z))_{\hat{\rho}_1}), $$
\noindent
since $g^1$ is fixed by the action of $\pi_1 \baff$ defined by
$\hat{\rho}_1$. The following two propositions are stated below
without proof.

\begin{prop}[Proposition 2.1 \cite{charlap}]\label{prop:charlap_1}
  Let $x \in \mathrm{E}_2^{p,1}$. Then
  \begin{equation*}
    x = \theta(x) \cdot f^1,
  \end{equation*}
  \noindent
  where $\cdot$ denotes the auxiliary pairing defined above.
\end{prop}

\begin{prop}[Proposition 2.2 \cite{charlap}] \label{prop:charlap_2}
  Let $x \in \mathrm{E}_2^{p,1}$. Then
  $$ \de^{(2)} (x) = (-1)^{p+1} \theta(x) \cdot \hat{\de}^{(2)}(f^1). $$
\end{prop}

Proposition \ref{prop:charlap_2} reduces the problem of determining
$\de^{(2)}$ to that of determining $\hat{\de}^{(2)}(f^1)$, which
depends on the universal Chern class $c_U$: this 
is the content of Theorem \ref{thm:diff_general}. \\

Before computing
$\hat{\de}^{(2)}(f^1)$, it is useful to compute 
the corresponding differential for the bundle
\begin{equation}
  \label{eq:88}
  \rzn \hookrightarrow \mathrm{E}\rzn \to \mathrm{B}\rzn,
\end{equation}
\noindent
which, in light of Lemma \ref{lemma:commutative_diagram}, is
isomorphic to the pull-back 
of the universal Lagrangian bundle along the universal
covering map $ \tau: \brzn \to \baff$. Let
$$ \hat{\mathrm{E}}^{p,q}_{2,\rzn} \cong
\mathrm{H}^p(\mathrm{B}\rzn;\mathrm{H}^q(\rzn;\mathrm{H}_1(\rzn;\Z))) $$
\noindent
denote the $\mathrm{E}_2$-page of the Leray-Serre spectral sequence for
the bundle (\ref{eq:88}) with $\mathrm{H}_1(\rzn;\Z)$ coefficients, and
let $\hat{\de}^{(2)}_{\rzn}$ denote the corresponding differential. Let
$f^1_{\rzn} \in  \hat{\mathrm{E}}^{0,1}_{2,\rzn}$ be the element
arising from the identity map
$$ \text{id} : \mathrm{H}_1(\rzn;\Z) \to  \mathrm{H}_1(\rzn;\Z) $$
\noindent
as above. In fact, in light of Lemma \ref{lemma:commutative_diagram},
the relation 
between $f^1$ and $f^1_{\rzn}$ is given by
$$ f^1_{\rzn} = \tau^* f^1.$$
\noindent
Let
\begin{equation}
  \label{eq:89}
  \psi_{\rzn}: \mathrm{H}^2(\mathrm{B}\rzn;\mathrm{H}_1(\rzn;\Z)) \to
  \hat{\mathrm{E}}^{2,0}_{2,\rzn} 
\end{equation}
\noindent
be the isomorphism arising via an identification
$$ \mathrm{H}^0(\rzn;\mathrm{H}_1(\rzn;\Z)) \cong
\mathrm{H}_1(\rzn;\Z). $$
\noindent
The following lemma computes the image of $f^1_{\rzn}$ under the
differential $\hat{\de}^{(2)}_{\rzn}$. This is a well-known result
(\textit{e.g.} \cite{geiges}), but a proof is included for
completeness.
\begin{lemma}\label{lemma:main_result_torus}
  $$\hat{\de}^{(2)}_{\rzn} (f^1_{\rzn}) = \psi(c_{\rzn}), $$
  \noindent
  where $c_{\rzn} \in
  \mathrm{H}^2(\mathrm{B}\rzn;\mathrm{H}_1(\rzn;\Z))$ is the Chern class
  of the universal bundle $\rzn \hookrightarrow \mathrm{E}\rzn \to
  \mathrm{B}\rzn$.
\end{lemma}
\begin{proof}
  By functoriality of the Chern class, it suffices to prove the result
  when $n=1$. In this case, the 
  universal bundle is isomorphic (up to homotopy) to
  \begin{equation}
    \label{eq:90}
    S^1 \hookrightarrow S^{\infty} \to \mathbb{C}\mathrm{P}^{\infty}.
  \end{equation}
  \noindent
  Since $S^{\infty}$ is contractible, the differential 
  $$\hat{\de}^{(2)}_{S^1} : \hat{\mathrm{E}}^{0,1}_{2,S^1} \to
  \hat{\mathrm{E}}^{2,0}_{2,S^1} $$
  \noindent
  is an isomorphism. In particular, since $f^1_{S^1} \in
  \hat{\mathrm{E}}^{0,1}_{2,S^1}$ is a generator,
  $\hat{\de}^{(2)}_{S^1} (f^1_{S^1})$ is a generator of
  $\hat{\mathrm{E}}^{2,0}_{2,S^1}$, and thus equal to $\pm
  \psi_{S^1}(c_{S^1})$. In fact, the normalisation axiom for the
  Chern classes implies that 
  $$ \hat{\de}^{(2)}_{S^1} (f^1_{S^1}) =  \psi_{S^1}(c_{S^1}).$$
\end{proof}

Let $\mathrm{E}^{*,*}_{2,\rzn},~\bar{\mathrm{E}}^{*,*}_{2,\rzn}$ denote the
$\mathrm{E}_2$-pages of the Leray-Serre spectral sequences of 
$$\rzn
\hookrightarrow \mathrm{E}\rzn \to \mathrm{B}\rzn$$
\noindent
with $\Z$ and
$\mathrm{H}^1(\rzn;\Z)$ coefficients respectively, and denote by
$\detwo_{\rzn},~\hat{\de}^{(2)}_{\rzn}$ their respective
differentials. Lemma \ref{lemma:main_result_torus} and Proposition
\ref{prop:charlap_2} can be 
combined to prove the following corollary, which is a version of
Theorem \ref{thm:diff_general} for principal $\rzn$-bundles.

\begin{cor} \label{cor:diff_torus}
  Let $x \in \mathrm{E}^{p,1}_{2,\rzn}$. Its image $\detwo_{\rzn}(x) \in
  \E^{p+2,0}_{2,\rzn}$ is given by
  $$ \detwo_{\rzn}(x) = (-1)^{p+1} \theta_{\rzn}(x) \cdot_{\rzn}
  \psi_{\rzn}(c_{\rzn}), $$
  \noindent
  where $\theta_{\rzn},~\cdot_{\rzn}$ are the analogues of the
  isomorphism of equation (\ref{eq:86}) and the auxiliary pairing
  of equation (\ref{eq:84}) for $\rzn \hookrightarrow \E
  \rzn \to \mathrm{B}\rzn$.
\end{cor}

\subsubsection*{The differential $\detwo$}\label{sec:differential-detwo}
The following theorem is one of the main results of this paper; it
states that, up to isomorphisms, the differential
$$ \detwo : \E^{p,1}_2 \to \E^{p+2,0}_2 $$
\noindent
is given by taking the cup product with the universal
Chern class $c_U$.
 
\begin{thm}\label{thm:diff_general}
  Let $x \in \E^{p,1}_2$. Its image $\detwo (x) \in \E^{p+2,0}_2 \cong
  \mathrm{H}^{p+2}(\baff;\Z)$ is given by
  \begin{equation}
    \label{eq:91}
    \detwo (x) = (-1)^{p+1}\theta(x) \cdot \psi(c_U),
  \end{equation}
  \noindent
  where $\psi : \mathrm{H}^2(\baff;\Z^n_{\mathrm{id}_*}) \to
  \hat{\E}^{2,0}_2$ is the isomorphism induced by the identification
  $$ \mathrm{H}^0(\rzn;\mathrm{H}_1(\rzn;\Z)) \cong \Z^n. $$
\end{thm}
\begin{proof}
  In light of Proposition \ref{prop:charlap_2}, it suffices to show that
  \begin{equation}
    \label{eq:92}
    \hat{\de}^{(2)}(f^1) = \psi (c_U),
  \end{equation}
  \noindent
  which is just the equivariant version of the result of Lemma
  \ref{lemma:main_result_torus}. The idea of the proof is to use Lemma
  \ref{lemma:main_result_torus} and functoriality of the Leray-Serre
  spectral sequence to deduce the result. \\
  
  By Lemma \ref{lemma:commutative_diagram}, the pull-back bundle
  $$ \rzn \hookrightarrow \tau^*\bgl \to \brzn $$
  \noindent
  is a universal bundle $\rzn
  \hookrightarrow \E\rzn \to \B\rzn$. In particular, there is a
  commutative diagram
  \begin{equation}
    \label{eq:93}
    \xymatrix{\hat{\E}^{0,1}_2 \ar[r]^-{\tau^*}
      \ar[d]_-{\hat{\de}^{(2)}} & \hat{\E}^{0,1}_{2,\rzn}
      \ar[d]^-{\hat{\de}^{(2)}_{\rzn}} \\
      \hat{\E}^{2,0}_2 \ar[r]_-{\tau^*} & \hat{\E}^{2,0}_{2,\rzn}}
  \end{equation}
  \noindent
  arising from functoriality of the Leray-Serre spectral sequence (cf.
  Section 6 in \cite{mccleary}). By Lemma \ref{lemma:main_result_torus},
  \begin{equation}
    \label{eq:94}
    \hat{\de}^{(2)}_{\rzn} \circ \tau^* (f^1) = \psi_{\rzn}(c_{\rzn}),
  \end{equation}
  \noindent
  since $f^1_{\rzn} = \tau^* f^1$. Equation (\ref{eq:94}) and the
  commutativity of the diagram in equation (\ref{eq:93}) imply that
  \begin{equation}
    \label{eq:95}
    \tau^*\circ \hat{\de}^{(2)} (f^1) = \psi_{\rzn}(c_{\rzn}).
  \end{equation}
  \noindent
  Note that, by definition, the isomorphism $\psi$ is the
  \emph{equivariant} version of $\psi_{\rzn}$, \textit{i.e.} there is a
  commutative diagram
  \begin{equation}
    \label{eq:96}
    \xymatrix{\mathrm{H}^2(\baff;\Z^n_{\text{id}_*}) \ar[r]^-{\psi}
      \ar[d]_-{\tau^*} & \hat{\E}^{2,0}_2 \ar[d]^-{\tau^*} \\
      \mathrm{H}^2(\brzn;\Z^n) \ar[r]_-{\psi_{\rzn}} & \hat{\E}^{2,0}_{2,\rzn}.}
  \end{equation}
  \noindent
  By equation (\ref{eq:82}), the commutative diagram (\ref{eq:96})
  implies that
  \begin{equation}
    \label{eq:97}
    \tau^* \circ \psi(c_U) = \psi_{\rzn}(c_{\rzn}).
  \end{equation}
  \noindent
  In particular, combining equations (\ref{eq:95}) and (\ref{eq:97}),
  obtain that
  \begin{equation}
    \label{eq:98}
    \hat{\de}^{(2)} (f^1) - \psi(c_U) = \mu \in \ker \tau^*.
  \end{equation}
  \noindent
  It therefore remains to show that
  $\mu = 0$. \\

  This is achieved in the following two steps.
  \begin{enumerate}[label= \textbf{Step \arabic*}, ref=Step \arabic*]
  \item \label{item:38} Prove that $\psi^{-1}(\mu)$ lies in the image of the
    homomorphism 
    $$\mathrm{Lin}^* : \Hh^2(\bgl;\Z^n_{\mathrm{id}^G_*}) \to \Hh^2(\baff;
    \Z^n_{\mathrm{id}_*}) $$
    \noindent
    induced by the fibration
    $$ \mathrm{Lin} : \baff \to \bgl; $$
  \item \label{item:39} Prove that $\psi^{-1}(\mu)$ lies in the kernel of the
    homomorphism
    $$ \sigma^* : \Hh^2(\baff;\Z^n_{\mathrm{id}_*}) \to
    \Hh^2(\bgl;\Z^n_{\mathrm{id}^G_*}) $$
    \noindent
    induced by the universal Lagrangian bundle
    $$ \univlf. $$
  \end{enumerate}
  \subsubsection*{\ref{item:38}} Recall that there is a fibration
  \begin{equation*}
    \xymatrix@1{\brzn \,\ar@{^{(}->}[r]^-{\tau} & \baff
      \ar[r]^-{\mathrm{Lin}} & \bgl} 
  \end{equation*}
  \noindent
  arising from the group $\affz(\rzn) = \glnz \ltimes \rzn$
  (cf. Section 5 in \cite{huse}), and
  consider the bundle
  \begin{equation}
    \label{eq:100}
    \rzn \hookrightarrow \sigma^* \bgl \to \bgl
  \end{equation}
  \noindent
  obtained by pulling back the universal Lagrangian
  bundle along the map $\sigma : \bgl \to \baff$ induced by the
  splitting $\sigma : \glnz \to \affz(\rzn)$. There is an associated
  system of local coefficients 
  \begin{equation}
    \label{eq:101}
    \Hh^1(\rzn;\Z) \hookrightarrow \sigma^* P_n \to \bgl,
  \end{equation}
  \noindent
  which is just the pull-back of the universal period lattice bundle
  $P_n \to \baff$ along $\sigma$. The pull-back
  $$ \Hh^1(\rzn;\Z) \hookrightarrow \mathrm{Lin}^* \sigma^* P_n \to
  \baff $$
  \noindent
  is classified by the (conjugacy class of the) homomorphism
  \begin{equation*}
    \sigma_* \circ \mathrm{Lin}_* : \pi_1(\baff) \to \pi_1(\baff),
  \end{equation*}
  \noindent
  which is the identity, as the composition of homomorphisms
  $$ \sigma \circ \mathrm{Lin} : \aff (\rzn) \to \aff(\rzn) $$
  \noindent
  preserves path-connected components.  Hence, the following cohomology
  rings are isomorphic
  \begin{equation}
    \label{eq:155}
    \Hh^*(\baff;\Z^n_{\mathrm{id}_*}) \cong \Hh^*(\baff;
    \Hh_1(\rzn;\Z)_{(\sigma \circ \mathrm{Lin})_*}).
  \end{equation}
  \noindent
  There is a
  twisted version of the Leray-Serre spectral sequence, constructed by
  Siegel in \cite{siegel}, which allows to calculate the cohomology
  ring 
  $$\Hh^*(\baff;\Hh_1(\rzn;\Z)_{(\sigma \circ \mathrm{Lin})_*})$$
  \noindent
  using the fibration
  $$ \xymatrix@1{\brzn \,\ar@{^{(}->}[r]^-{\tau} & \baff
    \ar[r]^-{\mathrm{Lin}} & \bgl} $$
  \noindent
  and the system of local coefficients on
  $\bgl$ of equation (\ref{eq:101}). Let $\breve{\E}^{p,q}_2$ denote
  the $\E_2$-page of this spectral sequence; since $\brzn$ is simply
  connected, there is an exact sequence
  \begin{equation*}
    \xymatrix@1{ 0 \ar[r]& \breve{\E}_2^{2,0}
      \ar[r]^-{\mathrm{Lin}^*} &
      \mathrm{H}^2(\baff;\Z^n_{\text{id}_*}) \ar[r]^-{\tau^*} &
      \breve{\E}^{0,2}_2,}
  \end{equation*}
  \noindent
  where the isomorphism of equation (\ref{eq:155}) is used
  tacitly. There are isomorphisms 
  \begin{equation*}
    \begin{split}
      \breve{\E}^{2,0}_2 & \cong \Hh^2(\bgl;\Z^n_{\text{id}^G_*}), \\
      \breve{\E}^{0,2}_2 & \cong [\Hh^2(\brzn;\Z^n)]^G \subset \Hh^2(\brzn;\Z^n), 
    \end{split}
  \end{equation*}
  \noindent
  where $\text{id}^G_*: \pi_1 (\bgl) \cong \glnz \to
  \mathrm{Aut}(\Z^n)\cong \glnz$ is the identity and
  $[\,.\,]^G$ denotes the group of $\glnz$-invariant elements. Since $\ker
  \tau^* = \mathrm{im}\,\mathrm{Lin}^*$, it follows that $ \psi^{-1} (\mu) \in
  \mathrm{im}\,p^*$. This completes the proof of \ref{item:38}.

  \subsubsection*{\ref{item:39}} Consider the bundle of equation
  (\ref{eq:100}). This 
  bundle admits a section, which is induced by the identity
  map $\bgl \to \bgl$. Thus
  \begin{equation}
    \label{eq:102}
    \sigma^*(c_U) = 0 \in \Hh^2(\bgl;\Z^n_{\text{id}^G_*}).
  \end{equation}
  \noindent
  Moreover, if $\hat{\E}^{*,*}_{2,G}$ denotes the Leray-Serre spectral
  sequence for 
  $$\rzn \hookrightarrow \sigma^* \bgl \to \bgl$$ 
  \noindent
  with
  $\Hh_1(\rzn;\Z)$ coefficients, the differential
  $$ \hat{\de}_G^{(2)} : \hat{\E}^{0,1}_{2,G} \to
  \hat{\E}^{2,0}_{2,G} $$
  \noindent
  vanishes identically, as the bundle admits a section. Thus there is the
  following commutative diagram arising from functoriality of spectral
  sequences
  $$ \xymatrix{\hat{\E}^{0,1}_2 \ar[r]^-{\sigma^*}
    \ar[d]_-{\hat{\de}^{(2)}} & \hat{\E}^{0,1}_{2,G}
    \ar[d]^-{\hat{\de}^{(2)}_G \equiv 0} \\
    \hat{\E}^{2,0}_2 \ar[r]_-{\sigma^*} & \hat{\E}^{2,0}_{2,G},} $$
  \noindent
  which implies that
  \begin{equation}
    \label{eq:103}
    \sigma^* \hat{\de}^{(2)}(f^1) = 0.
  \end{equation}
  \noindent
  If $\psi_G : \Hh^2(\bgl; \Z^n_{\text{id}^G_*}) \to
  \hat{\E}^{2,0}_{2,G}$ is the isomorphism arising from the
  identification
  $$ \Hh^0(\rzn;\Hh_1(\rzn;\Z)) \cong \Z^n, $$
  \noindent
  there is a commutative diagram (cf. equation (\ref{eq:96}))
  \begin{equation}
    \label{eq:57}
    \xymatrix{\Hh^2(\baff;\Z^n_{\text{id}_*}) \ar[r]^-{\psi}
      \ar[d]_-{\sigma^*} & \hat{\E}^{2,0}_2 \ar[d]^-{\sigma^*} \\
      \Hh^2(\bgl;\Z^n_{\text{id}^G_*}) \ar[r]_-{\psi_G}
      & \hat{\E}^{2,0}_{2,G}.}
  \end{equation}
  \noindent
  Equation (\ref{eq:102}) and the diagram (\ref{eq:57}) imply that
  \begin{equation}
    \label{eq:104}
    \sigma^* \psi^* (c_U) = 0.
  \end{equation}
  \noindent
  Applying $\sigma^*$ to both sides of equation (\ref{eq:98}), and
  using equations (\ref{eq:103}) and (\ref{eq:104}), obtain that
  $$ \sigma^* \mu =0,$$
  \noindent
  which proves \ref{item:39}. \\

  By \ref{item:38}, there exists 
  $$\nu
  \in \Hh^2(\bgl;\Z^n_{\text{id}^G_*})$$
  such that $\psi^{-1} (\mu) = \mathrm{Lin}^* \nu$. , then
  \begin{equation}
    \label{eq:64}
    \sigma^* \circ \psi \circ \mathrm{Lin}^* (\nu) = \sigma^* \mu = 0,
  \end{equation}
  \noindent
  since $\mu \in \ker \sigma^*$ by \ref{item:39}. Commutativity of
  the diagram in equation (\ref{eq:57}) implies that 
  $$ \sigma^* \circ \psi \circ \mathrm{Lin}^* (\nu) = \psi_G \circ \sigma^*
  \circ \mathrm{Lin}^* (\nu). $$
  Since $\psi_G$ is an isomorphism, equation (\ref{eq:64}) implies that 
  $$ \sigma^* \circ \mathrm{Lin}^* (\nu) = 0; $$
  \noindent
  as $\sigma^* \circ \mathrm{Lin}^*$ is the identity on
  $\Hh^2(\bgl;\Z^n_{\text{id}^G_*})$, it follows that $\nu = 0$. Therefore,
  $$\mu =0$$
  \noindent
  as required.
\end{proof}

\subsection{Relation to almost Lagrangian
  bundles} \label{sec:relat-almost-lagr}
Throughout this section, fix an integral affine manifold
$\am$ with linear holonomy 
$\mathfrak{l}$ and an almost Lagrangian bundle $\alf$ with Chern class
$$c \in \Hh^2(B;\Z^n_{\mathfrak{l}^{-T}}).$$
\noindent
The aim of this section is to use Theorem \ref{thm:diff_general} to
compute the obstruction for the above bundle to be Lagrangian; Theorem
\ref{thm:realisability} proves that this obstruction is given by the
cohomology class of the
cup product of the Chern class $c$ with the cohomology class
of the symplectic form on the total space of the symplectic reference bundle
associated to $\am$ (cf. Lemma \ref{lemma:coho_class_omega_0}). 

\subsubsection*{Symplectic reference
  bundles} \label{sec:sympl-refer-lagr} 
The symplectic form $\omega_0$ on the total space of the
symplectic reference bundle 
$$ \rlf $$
\noindent
defines a cohomological invariant of $\am$.

\begin{lemma}\label{lemma:coho_class_omega_0}
  The 2-form $\omega_0$ defines a cohomology class
  $$ w_0 \in \Hh^1(B;\Hh^1(\rzn;\R)_{\mathfrak{l}}), $$
  \noindent
  where $\mathfrak{l}: \pi_1(B) \to \glnz \subset \mathrm{GL}(n;\R)$.
\end{lemma}
\begin{proof}
The cohomology theory used throughout this proof is \v Cech-de
Rham. The cohomology class of a closed differential
form can be represented as the obstruction to finding a globally
defined potential. Let $\mathcal{U}= \{U_{\alpha}\}$ be a good open
cover by integral affine coordinate neighbourhoods of $\am$. There exist local
action-angle coordinates
$(\mathbf{a}_{\alpha},\boldsymbol{\theta}_{\alpha})$ on
$\pi^{-1}(U_{\alpha}) \cong U_{\alpha} \times \rzn$, so that
\begin{equation}
  \label{eq:106}
   \omega_{\alpha} = \omega_{0}|_{\pi^{-1}(U_{\alpha})} = 
   \sum\limits_{i} \de a^i_{\alpha} \wedge \de \theta^i_{\alpha} = \de
   \Bigg(\sum\limits_{i} a^i_{\alpha} \de \theta^i_{\alpha}\Bigg),
\end{equation}
\noindent
The transition
functions $\varphi_{\beta \alpha}$ for this choice of local
trivialisations of $\cotan B/P_{\am} 
\to B$ are given by 
\begin{equation}
  \label{eq:74}
  \varphi_{\beta \alpha} (\mathbf{a}_{\alpha},\boldsymbol{\theta}_{\alpha}) =
  (A_{\beta \alpha} \mathbf{a}_{\alpha} + \mathbf{d}_{\beta \alpha},
  A_{\beta \alpha}^{-T} \boldsymbol{\theta}_{\alpha}),
\end{equation}
\noindent
where $A_{\beta \alpha} \in \glnz$ and $\mathbf{d}_{\beta \alpha} \in
\R^n$ is constant. For each $\alpha$, set
$$ \nu_{\alpha} = \sum\limits_{i} a^i_{\alpha} \de
\theta^i_{\alpha}.$$
\noindent
The cohomology class of $\omega_0$ (as a differential form on $\cotan
B/P_{\am}$) is given in \v Cech cohomology by
the cocycle
$$\tau_{\beta \alpha} = \varphi^*_{\beta \alpha} \nu_{\beta} -
\nu_{\alpha} = \sum\limits_{i,k} d^i_{\beta \alpha} \de \theta^i_{\beta}.$$
\noindent
Since $\mathcal{U}$ is a good cover for $B$, $\tau=\{\tau_{\beta
  \alpha}\}$ defines a one 
dimensional cohomology class $w_0$ on $B$ with coefficients in the
local coefficient system
$$ \Hh^1(\rzn;\R) \hookrightarrow P^* \to B,$$
\noindent
whose monodromy is given by $\mathfrak{l}: \pi_1(B) \to \glnz \subset
\mathrm{GL}(n;\R)$, since this equals the inverse transposed of the
monodromy of the period lattice bundle $P_{\am}$. This proves the result.
\end{proof}

\begin{rk}\label{rk:on_coho_class}
  If $\omega'$ is any other symplectic form on
  $\cotan B/P_{\am}$ making the topological reference bundle
  $\rzn \hookrightarrow \cotan B/P_{\am} \to B$ Lagrangian, then there
  exists a 2-form $\mu$ on $B$ such that
  $$ (\cotan B/P_{\am}, \omega') \qquad \text{and} \qquad (\cotan
  B/P_{\am}, \omega_0 + \pi^* \mu).$$ 
  \noindent
  are fibrewise symplectomorphic. In particular, $\mu = s^* \omega'$,
  where $s: B \to \cotan B/P_{\am}$ is any section. Therefore $\omega'
  - \pi^* s^* \omega'$ also defines the cohomology class $w_0$, and
  the construction is independent of the choice of section $s$.
\end{rk}

The cohomology class $w_0$ does not depend solely on the linear
holonomy $\mathfrak{l}$ of $\am$, as the next example illustrates
(cf. Example \ref{ex:symp_ref_bundles}).

\begin{ex}\label{ex:omega_0_different_examples}
  Let $\R/\Z$ and $\R/2\Z$ be the integral affine manifolds of Example
  \ref{ex:symp_ref_bundles}. These affine manifolds have trivial
  linear holonomy. Let $\omega_1$ and $\omega_2$ be
  symplectic forms on their respective symplectic reference
  bundles. These bundles are isomorphic as affine $\R/\Z$-bundles and
  their total spaces can be identified with $S^1 \times S^1$. The
  cohomology classes
  $$ w_{0,1}, w_{0,2} \in \Hh^1(S^1;\Hh^1(\R/\Z;\R)) \hookrightarrow
  \Hh^2(S^1 \times S^1; \R) $$
  \noindent
  defined from $\omega_1$ and $\omega_2$ as in Lemma
  \ref{lemma:coho_class_omega_0} satisfy
  $$ w_{0,2} = 2 w_{0,1}. $$
\end{ex}

In particular, the above example hints at the fact that $w_0$ is
an \emph{integral affine invariant} of the manifold $\am$. This is
evident from the cocycle $\tau = \{\tau_{\beta \alpha}\}$ representing
$w_0$ in the proof of Lemma \ref{lemma:coho_class_omega_0}, since it
depends on the translational components of the changes of integral
affine coordinates of $\am$. 

\begin{rk}\label{rk:why_real_coefficients}
  It is important to notice that the differential 1-forms $\de
  \theta_{\alpha}^1,\ldots, \de \theta_{\alpha}^n$ represent, in fact,
  integral cohomology classes in 
  $$\Hh^1(\rzn;\R) \cong \Hh^1(\rzn;\Z) \otimes_{\Z} \R $$
  \noindent
  for all indices $\alpha$. This is because these
  forms are dual to the flows of the vector fields $\partial/\partial
  \theta^1_{\alpha}, \ldots , \partial/\partial
  \theta^n_{\alpha}$ from time 0 to time 1. These curves define a
  basis of the integral homology groups 
  $$ \Hh_1(\rzn;\Z) $$
  \noindent
  of the fibres. Thus the reason why real
  coefficients are used throughout is that the translational
  components of the changes of integral affine coordinates of $\am$
  are not necessarily integral (cf. Remark
  \ref{rk:affine_subbundle}). 
\end{rk}

\subsubsection*{Realisability theorem}\label{sec:real-theor}
Let $\E^{*,*}_{2,B}$ denote the $\E_2$-page of the Leray-Serre 
spectral sequence with integer coefficients of the almost Lagrangian
bundle
$$\alf$$
\noindent
fixed above. Theorem \ref{thm:diff_general} and naturality of the
Leray-Serre spectral sequence imply the following corollary.

\begin{cor}\label{cor:diff_alb}
  Let $x \in \E^{p,1}_{2,B}$. Its image under the differential
  $\detwo_B$ is given by
  $$ \detwo_B (x) = (-1)^{p+1}\theta_B(x) \cdot_B \psi_B(c), $$
  \noindent
  where $\theta_B,~\psi_B,~\cdot_B$ are the pull-backs of
  $\theta,~\psi,~\cdot$ defined for the universal
  Lagrangian bundle.
\end{cor}

The main idea of Theorem \ref{thm:realisability} is to study the differential
$$ \detwo_{B,\R} : \E^{1,1}_{2,B,\R} \to  \E^{3,0}_{2,B,\R} $$
\noindent
on the $\E_2$-page of the Leray-Serre spectral sequence with real
coefficients (hence the subscript $\R$) associated to the above almost
Lagrangian bundle. Corollary
\ref{cor:diff_alb} can be used to show that if $x \in \E^{p,1}_{2,B,\R}$,
then
$$ \detwo_{B,\R} (x) = (-1)^{p+1}\theta_{B,\R}(x) \cdot_{B,\R} \psi_{B,\R}(c^{\R}), $$
\noindent
where $c^{\R} \in \Hh^2(B;\R^n_{\chi_*})$ is the image of $c$ under
the homomorphism
\begin{equation}
  \label{eq:105}
  \Hh^2(B;\Z^n_{\mathfrak{l}^{-T}}) \to  \Hh^2(B;\R^n_{\mathfrak{l}^{-T}})
\end{equation}
\noindent
induced by the standard inclusion 
$$\Z^n \hookrightarrow \R^n \cong \Z^n \otimes_{\Z} \R,$$ 
\noindent
and $\theta_{B,\R},~\psi_{B,\R}$ are the appropriate
isomorphisms. Henceforth, the subscripts $B,~\R$ are omitted in order
to simplify notation. \\
 
First, note that
$$ w_0 \in \E^{1,1}_2, $$
\noindent
since, up to a choice of basepoints, $\mathfrak{l}$ equals the local
coefficient system 
$$ \rho_1 : \pi_1(B) \to \mathrm{Aut}(\Hh^1(\rzn;\R)) $$
\noindent
which defines
$$ \E^{1,1}_2 \cong \Hh^1(B;\Hh^1(\rzn;\R)_{\rho_1}). $$
\noindent
Secondly, using local action-angle coordinates, it is possible to give an
explicit cocycle representing the form $w_0$ in \v Cech-de Rham
cohomology. Let $\mathcal{U} = \{U_{\alpha}\}$ be a good open cover by
integral affine coordinate neighbourhoods of $\am$. Remark
\ref{rk:local_trivialisations_of_alb} shows that there exist local
trivialisations 
$$\Upsilon_{\alpha} : \pi^{-1}(U_{\alpha}) \to \cotan
U_{\alpha}/P_{\am}|_{U_{\alpha}}$$
\noindent
inducing action-angle coordinates
$(\mathbf{a}_{\alpha}, \boldsymbol{\theta}_{\alpha})$ on
$\pi^{-1}(U_{\alpha})$; the corresponding transition functions $\varphi_{\beta
  \alpha}$ are of the form
$$ \varphi_{\beta
  \alpha}(\mathbf{a}_{\alpha},\boldsymbol{\theta}_{\alpha}) =
(A_{\beta   \alpha} \mathbf{a}_{\alpha} + \mathbf{d}_{\beta \alpha}, A^{-T}_{\beta \alpha}
\boldsymbol{\theta}_{\alpha} + \mathbf{g}_{\beta
  \alpha}(\mathbf{a}_{\alpha})), $$
\noindent
where the first component corresponds to a change in integral affine
coordinates on $\am$. The family of 1-forms
$$ \bar{\tau}_{\beta \alpha} = \sum\limits_{i=1}^n d^i_{\beta \alpha} \de
\theta^i_{\beta} $$
\noindent
define an element in $\Hh^1(B;\Hh^1(\rzn;\R)_{\rho_1})$, since the
forms $\de \theta^1_{\beta},\ldots,\de \theta^n_{\beta}$ are closed
and the family of cohomology classes of $\bar{\tau}_{\beta \alpha}$ form a \v
Cech cocycle. This last statement holds since
$$ (\delta \bar{\tau})_{\alpha_1 \alpha_2 \alpha_3} = \de \Bigg(\sum\limits_{i=1}^n
d^i_{\alpha_2 \alpha_3} g^i_{\alpha_2 \alpha_1}\Bigg), $$
\noindent
where $\delta$ denotes \v Cech differential. Since the angle coordinates
$\boldsymbol{\theta}_{\alpha}$ are pulled back from angle coordinates on
the symplectic reference bundle via the local
trivialisations $\Upsilon_{\alpha}$, it follows that the cohomology
class that $\{\bar{\tau}_{\beta \alpha}\}$ defines equals the
cohomology class $w_0$, since for all indices $\alpha, \beta$
$$ \tau_{\beta \alpha} = \bar{\tau}_{\beta \alpha} $$
\noindent
(cf. Lemma \ref{lemma:coho_class_omega_0}). \\

With the above constructions in place, it is possible to prove the
main theorem of this paper.

\begin{thm}\label{thm:realisability}
  Let $(B, \mathcal{A})$ denote an integral affine manifold with
  linear holonomy $\mathfrak{l}$. An almost Lagrangian bundle $\alf$
  over $\am$ is a Lagrangian bundle if and only if 
  \begin{equation*}
    \detwo(w_0) = 0,
  \end{equation*}
  \noindent
  where $\detwo : \E^{1,1}_2 \to \E^{3,0}_2$ is the differential on
  the $\E^2$-page of the Leray-Serre spectral sequence for $\alf$ with
  real coefficients.
\end{thm}
\begin{proof}
  \v Cech-de Rham cohomology and the corresponding interpretation of
  the Leray-Serre spectral sequence with real coefficients
  are used throughout this proof. Firstly,
  suppose that the bundle is, in fact, Lagrangian. Let $\omega$ denote
  a symplectic form on $M$ making $\alf$ Lagrangian. Theorem 
  \ref{thm:ex_local_action_angle} and equation (\ref{eq:41})
  imply that there exists a good open cover $\mathcal{U}
  =\{U_{\alpha}\}$ of $B$ with local action-angle
  coordinates $(\mathbf{a}_{\alpha},\boldsymbol{\theta}_{\alpha})$ on each
  $\pi^{-1}(U_{\alpha})$ and transition functions
  \begin{equation*}
    \varphi_{\alpha_2
      \alpha_1}(\mathbf{a}_{\alpha_1},\boldsymbol{\theta}_{\alpha_1}) = (A_{\alpha_2 
      \alpha_1} \mathbf{a}_{\alpha_1} + \mathbf{d}_{\alpha_2
      \alpha_1}, A^{-T}_{\alpha_2 \alpha_1} 
    \boldsymbol{\theta}_{\alpha_1} + \mathbf{g}_{\alpha_2
      \alpha_1}(\mathbf{a}_{\alpha_1})),
  \end{equation*}
  \noindent
  with the functions $\mathbf{g}_{\alpha_2
    \alpha_1}$ constrained by
  \begin{equation}
    \label{eq:113}
    \varphi_{\alpha_2 \alpha_1}^*\omega_{\alpha_2} = \omega_{\alpha_1},
  \end{equation}
  \noindent
  where $\omega_{\alpha_i}$ denotes the restriction of $\omega$ to
  $\pi^{-1}(U_{\alpha_i})$ for $i=1,2$. On each intersection, equation
  (\ref{eq:113}) is equivalent to
  \begin{equation}
    \label{eq:67}
    \sum\limits_{i,j=1}^n A^{ij}_{\alpha_2 \alpha_1} \de a^j_{\alpha_1}
    \wedge \de g^i_{\alpha_2 \alpha_1} = 0;
  \end{equation}
  \noindent
  since each $U_{\alpha_1} \cap U_{\alpha_2}$ is simply-connected (in
  fact, contractible), 
  the forms $\de a^i_{\alpha_1}$ are exact and so the above equation
  implies that 
  \begin{equation*}
    \de \Bigg(\sum\limits_{i,j=1}^n A^{ij}_{\alpha_2 \alpha_1} a^j_{\alpha_1} \de
    g^i_{\alpha_2 \alpha_1}\Bigg) = 0.
  \end{equation*}
  \noindent
  As in the proof of Lemma \ref{lemma:coho_class_omega_0}, let 
  $$ \sigma_{\alpha} = \sum\limits_{i=1}^n a_{\alpha}^i \de
  \theta^i_{\alpha} $$
  \noindent
  be a locally defined potential for $\omega$. Since equation
  (\ref{eq:113}) holds, 
  $$ \kappa_{\alpha_2 \alpha_1} = \varphi^*_{\alpha_2 \alpha_1} \sigma_{\alpha_2} -
  \sigma_{\alpha_1} $$
  \noindent
  is a closed 1-form. The \v Cech cocycle $\{\kappa_{\alpha_2 \alpha_1}\}$
  represents the cohomology class of $\omega$ in $\Hh^2(M;\R)$. Take
  $$ \underbrace{\sum\limits_{i,j=1}^n A^{ij}_{\alpha_2 \alpha_1} a^j_{\alpha_1} \de
  g^i_{\alpha_2 \alpha_1}}_{:=\zeta_{\alpha_2 \alpha_1}}  +
\underbrace{\sum\limits_{i=1}^n d^i_{\alpha_2 \alpha_1} \de 
  \theta^i_{\alpha_2}}_{= \tau_{\alpha_2 \alpha_1}} $$
  as a representative of $\kappa_{\alpha_2 \alpha_1}$.
  \noindent
  Note that $\{\tau_{\alpha_2 \alpha_1}\}$ is a representative of the
  cohomology class $w_0 \in \E^{1,1}_2$; if $\delta$ denotes the \v Cech-boundary
  map, then
  $$ (\delta \tau)_{\alpha_1\alpha_2 \alpha_3} = \de
  \Bigg(\sum\limits_{i=1}^nd^i_{\alpha_2 \alpha_3} g^i_{\alpha_2 \alpha_1}\Bigg) := \de
  \eta_{\alpha_1 \alpha_2 \alpha_3}. $$
  \noindent
  The local potentials $\eta_{\alpha_1 \alpha_2 \alpha_3}$ are defined up to
  a choice of constants. The functions $\{-(\delta \eta)_{\alpha_1 \alpha_2
    \alpha_3 \alpha_4}\}$ are a \v Cech-de Rham cocycle whose
  corresponding cohomology class in $\E^{3,0}_2$ is
  $\detwo (w_0)$ since the cover $\mathcal{U}$ is good (cf. Section 9 in \cite{bott_tu}). \\
  
  The \v Cech boundary of $\{\zeta_{\alpha_2 \alpha_1}\}$ is equal to
  \begin{equation}
    \label{eq:157}
    (\delta \zeta)_{\alpha_1 \alpha_2 \alpha_3} = \de \Bigg(\sum\limits_{i=1}^n
    (d^i_{\alpha_1 \alpha_3} - d^i_{\alpha_1 \alpha_2}) g^i_{\alpha_1 \alpha_2}\Bigg):= \de
    \xi_{\alpha_1 \alpha_2 \alpha_3}.
  \end{equation}
  \noindent
  A simple calculation shows that for all indices $\alpha_1, \alpha_2, \alpha_3,
  \alpha_4$
  \begin{equation}
    \label{eq:115}
    (\delta \xi)_{\alpha_1 \alpha_2 \alpha_3 \alpha_4} = - (\delta
    \eta)_{\alpha_1 \alpha_2 \alpha_3 \alpha_4}.
  \end{equation}
  \noindent
  In particular, $\delta \xi$ is a
  \v Cech-de Rham cocycle representing the cohomology class $\detwo
  (w_0)$. In order to prove that this class vanishes, it suffices to show that
  $\xi_{\alpha_1 \alpha_2 \alpha_3}$ can be chosen so that 
  $$ (\delta \xi)_{\alpha_1 \alpha_2 \alpha_3 \alpha_4} = 0$$
  for all indices $\alpha_1, \alpha_2, \alpha_3, \alpha_4$. \\

  Since equation (\ref{eq:67}) holds and the cover $\mathcal{U}$ is
  good, the 1-forms $\zeta_{\alpha_2 \alpha_1}$ are closed and, hence,
  exact. Set
  $$ \zeta_{\alpha_2 \alpha_1} = \de \epsilon_{\alpha_2
    \alpha_1} $$
  \noindent
  for each pair of indices $\alpha_1, \alpha_2$. Then
  \begin{equation}
    \label{eq:156}
    (\delta \zeta)_{\alpha_1 \alpha_2 \alpha_3} = \de (\epsilon_{\alpha_2
      \alpha_3} - \epsilon_{\alpha_1 \alpha_3} + \epsilon_{\alpha_3 \alpha_2}) =
    \de (\delta \epsilon)_{\alpha_1 \alpha_2 \alpha_3}.
  \end{equation}
  \noindent
  Equations (\ref{eq:157}) and (\ref{eq:156}) imply that
  $$ \xi_{\alpha_1 \alpha_2 \alpha_3} = (\delta \epsilon)_{\alpha_1 \alpha_2
    \alpha_3} + C_{\alpha_1 \alpha_2 \alpha_3} $$
  \noindent
  for some constants $C_{\alpha_1 \alpha_2 \alpha_3}$. By substituting
  $\xi_{\alpha_1 \alpha_2 \alpha_3}$ with $\xi_{\alpha_1 \alpha_2 \alpha_3}-
  C_{\alpha_1 \alpha_2 \alpha_3}$, it may be assumed that
  $$  \xi_{\alpha_1 \alpha_2 \alpha_1} = (\delta \epsilon)_{\alpha_1 \alpha_2
    \alpha_3}, $$ 
  \noindent
  which, in turn, implies that for all indices $\alpha_1, \alpha_2, \alpha_3,
  \alpha_4$
  $$ (\delta \xi)_{\alpha_1 \alpha_2 \alpha_3 \alpha_4} = (\delta^2
  \epsilon)_{\alpha_1 \alpha_2 \alpha_3 \alpha_4} = 0. $$
  \noindent
  This proves that $\detwo (w_0) =
  0$. \\

  Conversely, suppose that $\detwo (w_0) = 0 $ for the almost
  Lagrangian bundle $\alf$. Let $\mathcal{U} = \{U_{\alpha}\}$ be the
  good cover of $B$ given by Remark
  \ref{rk:local_trivialisations_of_alb}, \textit{i.e.} there exist
  local action-angle coordinates $(\mathbf{a}_{\alpha},
  \boldsymbol{\theta}_{\alpha})$ on $\pi^{-1}(U_{\alpha})$ and the
  transition functions are of the form 
  $$ \varphi_{\alpha_2 
    \alpha_1}(\mathbf{a}_{\alpha_1},\boldsymbol{\theta}_{\alpha_1}) = (A_{\alpha_2 
    \alpha_1} \mathbf{a}_{\alpha_1} + \mathbf{d}_{\alpha_2 \alpha_1}, A^{-T}_{\alpha_2 \alpha_1}
  \boldsymbol{\theta}_{\alpha_1} + \mathbf{g}_{\alpha_2
    \alpha_1}(\mathbf{a}_{\alpha_1})), $$
  without the constraint on the functions $\mathbf{g}_{\alpha_2
    \alpha_1}$ given by equation (\ref{eq:113}). Recall that the form 
  $$ \omega_{\alpha} = \sum\limits_{i=1}^n \de a^i_{\alpha} \wedge \de
  \theta^i_{\alpha} $$
  \noindent
  makes the bundle $\rzn \hookrightarrow \pi^{-1}(U_{\alpha}) \to
  U_{\alpha}$ Lagrangian. The obstruction to patching these forms
  together to yield a globally defined symplectic form $\omega$ on $M$
  which makes the bundle $\alf$ Lagrangian is given by the \v Cech cocycle
  \begin{equation}
    \label{eq:159}
    \varphi^*_{\alpha_2 \alpha_1} \omega_{\alpha_2} - \omega_{\alpha_1} =
    \sum\limits_{i,j=1}^n A^{ij}_{\alpha_2 \alpha_1}\de a^j_{\alpha_1} \wedge
    \de g^i_{\alpha_2 \alpha_1}.
  \end{equation}
  \noindent
  Since the cover $\mathcal{U}$ is good, this cocycle represents a
  cohomology class in $\Hh^1(B;\mathcal{Z}^2(\cotan B))$, where
  $\mathcal{Z}^2(\cotan B)$ denotes the sheaf of closed sections of
  the bundle $\bigwedge^2 \cotan B \to B$. In light of the
  isomorphism 
  \begin{equation}
    \label{eq:2}
    \Hh^1(B;\mathcal{Z}^2(\cotan B)) \cong \Hh^3(B;\R)
  \end{equation}
  \noindent
  (cf. Theorem 8.1 in \cite{bott_tu}, \cite{daz_delz}), the above cocycle defines a
  cohomology class 
  $$ \upsilon \in \Hh^3(B;\R).$$
  \noindent
  Using the notation of the first half of the proof, a cocycle
  representing $\upsilon$ in \v Cech cohomology is given by $-\delta
  \xi$ (this simply unravels the isomorphism of equation (\ref{eq:2})). 
  The equality of equation
  (\ref{eq:115}) still holds, since it is not necessary to have that
  the transition functions $\varphi_{\alpha_2 \alpha_1}$ are
  symplectomorphisms in order to prove it. 
  Thus
  $$ - \delta \xi = \delta \eta. $$
  \noindent
  Note that the \v Cech-de Rham
  cocycle $-\delta \eta$ is a representative of $\detwo (w_0)$; by
  assumption, this vanishes and
  $$ \upsilon = 0.$$
  \noindent
  Therefore
  $$ \sum\limits_{i,j=1}^n A^{ij}_{\alpha_2 \alpha_1}\de a^j_{\alpha_1} \wedge
  \de g^i_{\alpha_2 \alpha_1} $$
  is a coboundary; thus there exists closed 2-forms $k_{\alpha}$
  defined on each $U_{\alpha}$ such that
  \begin{equation}
    \label{eq:158}
    (\delta k)_{\alpha_2 \alpha_1} = \sum\limits_{i,j=1}^n
    A^{ij}_{\alpha_2 \alpha_1}\de a^j_{\alpha_1} \wedge 
    \de g^i_{\alpha_2 \alpha_1}
  \end{equation}
  \noindent
  for all indices $\alpha_1, \alpha_2$. \\

  The forms 
  $$ \omega_{\alpha} - \pi^* k_{\alpha} $$
  \noindent
  defined on each $\pi^{-1}(U_{\alpha})$ patch together by virtue of
  equations (\ref{eq:159}) and (\ref{eq:158}). Denote the resulting 2-form on $M$ by
  $\omega$. It is closed and non-degenerate since each
  $\omega_{\alpha} - \pi^* k_{\alpha}$ is; moreover, the fibres of the
  bundle are Lagrangian submanifolds of $\sm$, since they are
  Lagrangian submanifolds of the relevant symplectic manifold
  $(\pi^{-1}(U_{\alpha}), \omega_{\alpha} - \pi^* k_{\alpha})$ and,
  hence, the result follows.
\end{proof}

\begin{rk}\label{rk:proof_realisability}
\mbox{}
\begin{enumerate}[label=\roman*)]
\item   Let $\alf$ be an almost Lagrangian bundle over
  the integral affine manifold $\am$ with Chern class $c$, and let
  $w_0$ be the cohomology
  class of Lemma \ref{lemma:coho_class_omega_0}. If 
  $$ \detwo : \E^{1,1}_2 \to \E^{3,0}_2$$
  \noindent
  denotes the differential on the $E^2$-page of the Leray-Serre
  spectral sequence of $\alf$ with real coefficients, Theorem
  \ref{thm:realisability} proves that
  $$ \detwo (w_0) = -\mathcal{D}_{\am}(c),$$
  \noindent
  where $\mathcal{D}_{\am}$ denotes the homomorphism of Dazord and
  Delzant (cf. Theorem \ref{thm:daz_delz}). This is because the
  cohomology class $\mathcal{D}_{\am} (c)$ is given by the \v
  Cech-de Rham cocycle 
  $$ \{\sum\limits_{i,j=1}^n A^{ij}_{\beta \alpha} \de a^j_{\alpha}
  \wedge \de g^i_{\beta \alpha}\}, $$
  \noindent
  whose cohomology class can also be represented by the \v Cech cocycle
  $- \delta \xi$ defined above. The relation of equation
  (\ref{eq:115}) proves the claim; 
\item Theorem \ref{thm:realisability} may be called a
  \emph{realisability} theorem, since it provides a way to determine
  which cohomology classes in
  $\mathrm{H}^2(B;\Z^n_{\mathfrak{l}^{-T}})$ can be realised as Chern
  classes of Lagrangian bundles over $\am$. The terminology comes from
  the theory of symplectic realisations of Poisson manifolds (cf. Chapter 8 in \cite{vaisman}).
\end{enumerate}

\end{rk}


\section{Relation to Integral Affine
  Geometry} \label{sec:relat-integr-affine}
\subsection{The radiance obstruction of an affine
  manifold} \label{sec:radi-obstr-an}
In this section the radiance obstruction of an affine manifold is
introduced, following the work of Goldman and Hirsch in
\cite{goldman}. 

\subsubsection*{Universal radiance obstruction}\label{sec:univ-radi-obstr}
Consider the group
$$ \aff (\R^n) = \mathrm{GL}(n;\R) \ltimes \R^n. $$
\noindent
and the associated standard exact sequence
\begin{equation}
  \label{eq:121}
  \xymatrix@1{0 \ar[r] & \R^n \ar[r]^-{\iota} & \aff(\R^n)
    \ar[r]^-{\mathrm{Lin}} & \mathrm{GL}(n;\R) \ar[r] & 1.}
\end{equation}
\noindent
Define
\begin{equation}
  \label{eq:117}
  \begin{split}
    \mathrm{Trans}:& \aff(\R^n) \to \R^n \\
    & (A,\mathbf{b}) \mapsto \mathbf{b},
  \end{split}
\end{equation}
\noindent
which satisfies
\begin{equation*}
  \mathrm{Trans}((A,\mathbf{b}) \cdot (A',\mathbf{b}')) =
  \mathrm{Trans}(A,\mathbf{b}) +
  \text{Lin}(A,\mathbf{b})\mathrm{Trans}(A',\mathbf{b}'),  
\end{equation*}
\noindent
for all $(A,\mathbf{b}), (A',\mathbf{b}') \in \aff (\R^n)$. Thus
$\mathrm{Trans}$ is a \emph{crossed homomorphism} and it defines an element 
$$ r_U \in \Hh^1(\aff(\R^n);\R^n_{\mathrm{Lin}}), $$
\noindent
where $\R^n_{\mathrm{Lin}}$ denotes $\R^n$ as an $\aff(\R^n)$-module
via the homomorphism $\mathrm{Lin}$.

\begin{defn}[Universal radiance obstruction
  \cite{goldman,goldman_orbits}]\label{defn:univ_radiance_obs} 
  The cohomology class $r_U$ is called the \emph{universal radiance obstruction}.
\end{defn}

Let $\am$ be an affine manifold with affine holonomy representation
$$\mathfrak{a} : \pib \to \aff(\R^n)$$
\noindent
(cf. Definition \ref{defn:aff_lin_holo}). Let $\mathfrak{l} = \Lin
\circ \mathfrak{a}$ 
denote its linear holonomy and fix $\Gamma = \pib$ throughout.

\begin{defn}[Group theoretic definition of radiance obstruction
  \cite{goldman}]\label{defn:group_theoretic_radiance}
  The cohomology class
  $$ r_{\am} = \mathfrak{a}^* r_U \in
  \Hh^1(\Gamma;\R^n_{\mathfrak{l}}) $$
  \noindent
  is called the \emph{radiance obstruction} of the affine manifold
  $\am$. An affine manifold $\am$ whose radiance obstruction $r_{\am}$
  vanishes is called a \emph{radiant} manifold.
\end{defn}

\begin{rk}[Characterisation of radiant manifolds
  \cite{goldman}]\label{rk:char_radiant_mflds} 
  An affine manifold $\am$ is radiant if and only if its affine
  holonomy representation $\mathfrak{a}$ can be chosen so that its
  image lies entirely in $\glnr \subset \aff(\R^n)$. Equivalently, an
  affine manifold $\am$ is radiant if and only if there exists an affine
  structure $\mathcal{A}'$ which is affinely diffeomorphic to the
  given one and whose changes of coordinates lie in $\gl(n;\R)$.
\end{rk}

\subsubsection*{The topology of the universal radiance
  obstruction}\label{sec:topol-univ-radi}
The universal radiance obstruction $r_U$ can also be described
topologically as follows. Let $\mathfrak{a} : \Gamma \to \aff (\R^n)$
be a representation of a 
discrete group $\Gamma$. The composition $\mathrm{Trans} \circ
\mathfrak{a}$ defines a crossed homomorphism which represents a
cohomology class 
$$ r_{\Gamma} \in \Hh^1(\Gamma;\R^n_{\Lin \circ \mathfrak{a}}).$$
\noindent
For any topological group $G$, let $G^{\delta}$ denote the group
endowed with the discrete topology. In light of equation
(\ref{eq:121}), there is a split short exact sequence
$$  \xymatrix@1{0 \ar[r] & (\R^n)^{\delta} \ar[r]^-{\iota} & \aff(\R^n)^{\delta}
  \ar@<1ex>[r]^-{\mathrm{Lin}} & \mathrm{GL}(n;\R)^{\delta}
  \ar@<1ex>[l]^-{\sigma} \ar[r] & 1.} $$
\noindent
Applying Theorem \ref{thm:subgroup} to the splitting $\sigma$, obtain
a bundle
$$ (\R^n)^{\delta} \hookrightarrow \B \mathrm{GL}(n;\R)^{\delta} \to
\B \aff(\R^n)^{\delta}, $$
\noindent
where
$$ \B \mathrm{GL}(n;\R)^{\delta} \simeq \E \aff(\R^n)^{\delta}
/\mathrm{GL}(n;\R)^{\delta}. $$
\noindent
Note that $\B \aff(\R^n)^{\delta}$ and $\B \mathrm{GL}(n;\R)^{\delta}$
are $\mathrm{K}(\aff(\R^n);1)$ and $\mathrm{K}(\mathrm{GL}(n;\R);1)$
respectively. Using the ideas of the proof of Theorem 
\ref{thm:universality}, the following lemma can be
proved.

\begin{lemma}\label{lemma:structure_bgl_baff}
  There exists a bundle isomorphism
  $$ \xymatrix{ \E \aff(\R^n)^{\delta}
    /\mathrm{GL}(n;\R)^{\delta} \ar[0,2]^-{\cong} \ar[1,1]& &  \E
    \aff(\R^n)^{\delta} \times_{\aff(\R^n)^{\delta}} (\R^n)^{\delta}
    \ar[1,-1] \\
  & \B \aff(\R^n)^{\delta}. &} $$
\end{lemma}

Let $\Gamma$ be a discrete topological group. Any representation
$$ \mathfrak{a} : \Gamma \to \aff (\R^n) $$
\noindent
factors through $\aff (\R^n)^{\delta}$; the
homomorphism $\mathfrak{a}$ induces a map (defined up to homotopy)
$$ \bar{\mathfrak{a}} : \B \Gamma \to \B \aff (\R^n)^{\delta}$$
\noindent
whose induced map on fundamental groups coincides with $\mathfrak{a}$
(up to a choice of basepoints). A lift
$$ \xymatrix{& \B \mathrm{GL}(n;\R)^{\delta} \ar[d]^-{\sigma} \\
  \B \Gamma \ar@{.>}[-1,1] \ar[r]_-{\bar{\mathfrak{a}}} & \B \aff
  (\R^n)^{\delta}} $$
\noindent
exists if and only if
$$ \bar{\mathfrak{a}}_* (\Gamma) \subset
\sigma_* (\pi_1(\B\mathrm{GL}(n;\R)^{\delta})), $$
\noindent
since the fibre of the projection $\sigma :\B
\mathrm{GL}(n;\R)^{\delta} \to \B \aff(\R^n)^{\delta}$ is
discrete. Note that the induced map
$$ \sigma_* : \pi_1(\B \mathrm{GL}(n;\R)^{\delta}) \cong
\mathrm{GL}(n;\R) \to \pi_1(\B \aff (\R^n)^{\delta}) \cong
\aff(\R^n) $$
\noindent
equals the standard splitting
\begin{equation*}
  \begin{split}
    \sigma: \gl(n;\R) &\to \aff (\R^n) \\
    A & \mapsto (A,\mathbf{0}).
  \end{split}
\end{equation*}
\noindent
by construction. Hence, the map $\bar{\mathfrak{a}}$ admits a lift if and
only if the representation
$$ \bar{\mathfrak{a}}_* : \pi_1 (\B \Gamma) \cong \Gamma \to \pi_1(\B \aff
(\R^n)^{\delta}) \cong \aff(\R^n) $$
\noindent
lies entirely within the image of $\sigma_*$. The latter statement is
true if, up to conjugation, the image of the representation $\mathfrak{a}$ lies
entirely in $\sigma(\glnr)$, which is true if and only if $r_{\Gamma} =
0$. In particular, letting $\Gamma = \aff (\R^n)^{\delta} $ and
$\mathfrak{a} : \aff(\R^n)^{\delta} \to \aff(\R^n)$ be the identity
homomorphism, the above discussion proves the following theorem.

\begin{thm}\label{thm:obstr_existence_section}
  The universal radiance obstruction $r_U$ is the obstruction to the
  existence of a section for the fibration
  $$ (\R^n)^{\delta} \hookrightarrow \B \glnr^{\delta} \to \B
    \aff(\R^n)^{\delta}. $$
\end{thm}

\subsubsection*{A geometric interpretation}\label{sec:geom-interpr}
It is possible to give a geometric interpretation to the radiance
obstruction of an affine manifold, which, thus far, is simply a
cohomological invariant of its fundamental group. Firstly, note that
the tangent bundle of an affine manifold can be 
endowed with the structure of a \emph{flat} affine bundle.

\begin{defn}[Flat affine bundle \cite{goldman}]\label{defn:flat_affine_bundle}
  Let $(F,\mathcal{B}) \hookrightarrow E \to N$ be an affine bundle,
  so that $(F,\mathcal{B})$ is an affine manifold and its structure
  group is $\aff (F,\mathcal{B})$. It is said to be \emph{flat} if it
  admits locally constant transition functions. 
\end{defn}

\begin{ex}[Tangent bundles of affine manifolds]\label{ex:flat_affine_bundles}
  The tangent bundle of an affine manifold $\aam$
  (thought of as a vector bundle) is naturally a flat linear bundle, since the
  standard transition functions are locally constant. There is,
  however, a different choice of flat affine bundle 
  structure that can be chosen on the tangent bundle of an affine
  manifold $\am$, defined in \cite{goldman}, which is more
  \emph{natural} from the point of view of affinely flat geometry. The
  transition functions for the local affine trivialisations of this
  other flat affine bundle structure are simply given by the affine
  changes of coordinates on the base space (cf. \cite{goldman}). This
  flat affine bundle is called the \emph{affine} tangent bundle of the
  affine manifold $\am$ and is denoted by $\T^{\aff}B \to B$.
\end{ex}

Fix an $n$-dimensional manifold $\am$ and let $\T^{\aff} B \to B$ be
its affine tangent bundle. Since its fibres are contractible, there
exists a section to this bundle, \textit{e.g.} the zero
section. However, it is not necessarily true that the affine tangent
bundle admits a \emph{flat} section.

\begin{defn}[Flat section of a flat affine bundle
  \cite{goldman}]\label{defn:flat_section} 
  Let $(F,\mathcal{B}) \hookrightarrow E \to N$ be a flat affine
  bundle. A section $s: N \to E$ is \emph{flat} if, for any local
  trivialisation $\pi^{-1}(U_{\alpha}) \to U_{\alpha} \times F$, the
  composite
  $$ \xymatrix@1{U_{\alpha} \ar[r]^-{s_{\alpha}} &
    \pi^{-1}(U_{\alpha}) \ar[r] & U_{\alpha} \times F
    \ar[r]^-{\mathrm{pr}_2} & F} $$
  \noindent
  is locally constant.
\end{defn}

Fix an affine manifold $\am$ and let $\T^{\aff} B \to B$ denote the
corresponding affine tangent bundle with transition functions
$f_{\beta \alpha}$. Since the transition functions of $\T^{\aff}B \to
B$ are constant on connected components, it is possible to construct
the bundle
\begin{equation}
  \label{eq:127}
  (\R^n)^{\delta} \hookrightarrow (\T^{\aff} B)^{\delta} \to B,
\end{equation}
\noindent
obtained by endowing the fibres with the discrete topology. Denote its
classifying map by
$$ \chi_{\aff} : B \to \B \aff (\R^n)^{\delta}. $$
\noindent
In light of Lemma \ref{lemma:structure_bgl_baff}, the bundle
of equation (\ref{eq:127}) is isomorphic to the pull-back
$$ \chi^*_{\aff} \B \glnr^{\delta} \to B. $$
\noindent
Therefore the cohomology class
$$ \chi^*_{\aff} r_U $$
\noindent
is the obstruction to the existence of a section for the bundle of
equation (\ref{eq:127}). \\

The homotopy class of $\chi_{\aff}$ is determined by its
induced map on fundamental groups, since $\B \aff (\R^n)^{\delta}$ is
a $\mathrm{K}(\aff(\R^n);1)$. By construction, this map equals (up to
conjugation) the
affine holonomy 
$$\mathfrak{a} : \pi_1(B) = \Gamma \to \aff (\R^n) $$
\noindent
of the affine manifold $\am$ (cf. \cite{aus}). Let
$$ \bar{\mathfrak{a}} : \B \Gamma \to  \aff (\R^n)^{\delta} $$
\noindent
be the map (defined up to homotopy) induced by the affine holonomy and
let $\chi_{\tilde{B}} : B \to \B \Gamma$ denote the classifying map for
the universal covering $\tilde{B} \to B$. Then (up to homotopy)
$$ \chi_{\aff} = \bar{\mathfrak{a}} \circ \chi_{\tilde{B}}; $$
\noindent
in particular, 
$$ \chi_{\aff}^* r_U= \chi_{\tilde{B}}^* \circ \bar{\mathfrak{a}}^*
r_U = \chi_{\tilde{B}}^* r_{\am}. $$
\noindent
The map $\chi_{\tilde{B}}^*$ is an isomorphism
on one dimensional cohomology with any coefficient system
(cf. \cite{goldman}) and thus
\begin{equation}
  \label{eq:128}
  r_{\am} = (\chi_{\tilde{B}}^*)^{-1} \circ \chi_{\aff}^* r_U.
\end{equation}
\noindent
\begin{rk}\label{rk:notation}
  By abuse of nomenclature and notation, the class $\chi_{\aff}^*r_U$ is
  henceforth also referred to as the radiance obstruction of $\am$ and
  denoted by $r_{\am}$. 
\end{rk}

With this identification, the radiance obstruction $r_{\am}$ is the
obstruction to the existence of a section to the bundle of equation
(\ref{eq:127}). By construction, a section for the aforementioned
bundle exists if and only if $\T^{\aff} B \to B$ admits a flat
section. Therefore the following theorem holds.

\begin{thm}[Goldman and Hirsch
  \cite{goldman}]\label{thm:obstruction_existence_flat_section}
  Let $\am$ be an $n$-dimensional affine manifold with linear holonomy
  $\mathfrak{l}$. The radiance obstruction 
  $$ r_{\am} \in \Hh^1(B;\R^n_{\mathfrak{l}}) $$
  \noindent
  is the obstruction to the existence of a flat section for the flat
  affine bundle $\T^{\aff} B \to B$.
\end{thm}

\begin{ex}[Inequivalent flat affine structures on $\T (S^1 \times
  \R)$]\label{ex:affine_structures_ts1}
  \mbox{}
  \begin{enumerate}[label=\roman*)]
  \item   Consider the $\Z$-action on $\R^2$ given by translations in a fixed
    direction $\mathbf{b}_0 \neq 0$; this action is free, proper and by affine
    transformations 
    on $\R^2$. Thus the manifold $\R^2/\Z$ is affine and its affine
    holonomy 
    $$\mathfrak{a}_{\R^2/\Z}: \pi_1(\R^2/\Z) \to \aff (\R^2) $$
    \noindent
    is defined on a generator $\gamma$ as 
    \begin{equation*}
      \mathfrak{a}_{\R^2/\Z}(\gamma) = (I,\mathbf{b}_0),
    \end{equation*}
    \noindent
    and extended by linearity. The crossed homomorphism $\mathrm{Trans}
    \circ \mathfrak{a}_{\R^2/\Z}$ defines a non-zero cohomology
    class 
    $$r_{\R^2/\Z} \in \Hh^1(\R^2/\Z;\R^2). $$
    \noindent
    In light of Theorem
    \ref{thm:obstruction_existence_flat_section}, the flat affine bundle
    $\T^{\aff} (\R^2/\Z) \to \R^2/\Z$ does not admit a flat section; 
  \item The inclusion 
  $$ \R^2\setminus \{\mathbf{0}\} \hookrightarrow \R^2 $$
  \noindent
  induces an affine structure on $\R^2 \setminus \{\mathbf{0}\}$ which
  has trivial affine holonomy. Therefore the affine tangent bundle 
  $$ \T^{\aff} (\R^2 \setminus \{\mathbf{0}\}) \to \R^2\setminus
  \{\mathbf{0}\} $$
  \noindent
  admits a flat section.
  \end{enumerate}
\end{ex}

\subsection{Relation to Lagrangian
  bundles} \label{sec:relat-lagr-fibr}
In this section, the
radiance obstruction $r_{\am}$ is related to the problem of
constructing Lagrangian bundles over $B$ inducing the integral affine
structure $\mathcal{A}$. Throughout the following, fix an integral affine
manifold $\am$ whose linear holonomy is denoted by $\mathfrak{l}$. \\

Recall that $w_0$ is the cohomology class of the symplectic form
$\omega_0$ of the symplectic reference bundle 
associated to $\am$ 
$$ \rlf. $$
\noindent
Let $\mathcal{C}^{\infty}(\T^{\aff} B)$ and $\mathcal{C}^{\infty}(\cotan \rzn)$
denote the sheaves of sections of the affine tangent bundle $\T^{\aff} B \to B$ and
of 1-forms on the fibres of the symplectic reference
bundle respectively. The
symplectic form $\omega_0$ defines an isomorphism of sheaves
\begin{equation}
  \label{eq:132}
  \mathcal{C}^{\infty}(\cotan \rzn) \cong \mathcal{C}^{\infty}(\T^{\aff}
  B).
\end{equation}
\noindent
In local action-angle coordinates
$(\mathbf{a}_{\alpha},\boldsymbol{\theta}_{\alpha})$, the above
isomorphism is given by
\begin{equation}
  \label{eq:131}
  \sum\limits_{i=1}^n h^i_{\alpha} \de \theta^i_{\alpha} \mapsto
  \sum\limits_{i=1}^n h^i_{\alpha} \frac{\partial}{\partial a^i_{\alpha}}, 
\end{equation}
\noindent
where $h^1_{\alpha}, \ldots, h^n_{\alpha}$ are smooth
functions. By equation (\ref{eq:131}), the 
isomorphism of equation (\ref{eq:132}) restricted to the subsheaf
$\mathcal{C}^{\infty}_{\mathrm{flat}}(\cotan \rzn)$ of \emph{locally
  constant} sections descends to an isomorphism of sheaves
$$\mathcal{C}^{\infty}_{\mathrm{flat}}(\cotan \rzn) \cong
\mathcal{C}_{\mathrm{flat}}^{\infty}(\T^{\aff} B). $$
\noindent
Since $\cotan \rzn$ admits a frame of closed forms, a locally constant
section is given in local action-angle coordinates by
$$ \sum\limits_{i=1}^n r^i_{\alpha} \de \theta_{\alpha}^i, $$
\noindent
where $r^1_{\alpha}, \ldots, r^n_{\alpha} \in \R$ are constant. Such sections
are in 1-1 correspondence with cohomology classes in $\Hh^1(\rzn;\R)$,
since the forms $\de \theta^1_{\alpha}, \ldots, \de \theta^n_{\alpha}$
induce a basis of $\Hh^1(\rzn;\R)$. Hence, the symplectic form
$\omega_0$ induces an isomorphism of sheaves
\begin{equation}
  \label{eq:135}
  \mathcal{P}^* \cong \mathcal{C}_{\mathrm{flat}}^{\infty}(\T^{\aff} B),
\end{equation}
\noindent
where $\mathcal{P}^*$ denotes the sheaves of sections of the local
coefficient system
$$ \Hh^1(\rzn;\R) \hookrightarrow P^* \to B $$
\noindent
associated to the symplectic reference bundle. This
isomorphism induces an isomorphism of cohomology groups
\begin{equation}
  \label{eq:51}
  \Phi: \Hh^*(B;\mathcal{P}^*) \to
  \Hh^*(B;\mathcal{C}^{\infty}_{\mathrm{flat}}(\T^{\aff} B)). 
\end{equation}
\noindent
Since both $\mathcal{P}^*$ and
$\mathcal{C}^{\infty}_{\mathrm{flat}}(\T^{\aff} B)$ are locally constant
sheaves, the above induces an isomorphism
$$ \Phi : \Hh^*(B;\Hh^1(\rzn;\R)_{\mathfrak{l}}) \to
\Hh^1(B;\R^n_{\mathfrak{l}}). $$
\noindent

\begin{thm}\label{thm:equality_forms}
  The cohomology class $w_0 \in \Hh^1(B;\Hh^1(\rzn;\R)_{\mathfrak{l}})$
  defined by the symplectic form $\omega_0$ as in Lemma
  \ref{lemma:coho_class_omega_0}, maps to the radiance obstruction
  $r_{\am} \in \Hh^1(B;\R^n_{\mathfrak{l}})$.
\end{thm}
\begin{proof}
  By Theorem \ref{thm:obstruction_existence_flat_section}, the
  radiance obstruction $r_{\am}$ is the obstruction to the existence
  of a flat section to $\T^{\aff} B \to B$. Let
  $\mathcal{U}=\{U_{\alpha}\}$ be a good open cover by integral affine
  coordinate neighbourhoods of $\am$ and let
  $$ \phi_{\alpha}: U_{\alpha} \to \R^n $$
  \noindent
  denote the coordinate map. The section 
  \begin{equation}
    \label{eq:133}
    \begin{split}
      s_{\alpha} : U_{\alpha} &\to \T^{\aff} U_{\alpha} \\
      x &\mapsto (x,-(D\phi_{\alpha}(x))^{-1}(\phi_{\alpha}(x)))
    \end{split}
  \end{equation}
  \noindent
  is flat (cf. \cite{goldman}). The
  collection
  $$ \hat{\tau} = \{\hat{\tau}_{\beta \alpha} \} := \{s_{\beta} -
  s_{\alpha} \} $$
  \noindent
  is a \v Cech cocycle which represents $r_{\am}$
  (cf. \cite{goldman}). \\

  Let $\mathbf{a}_{\alpha}$ denote affine coordinates on $U_{\alpha}$
  induced by $\phi_{\alpha}$ and, as usual, set 
  $$ \phi_{\beta} \circ \phi_{\alpha}^{-1} (\mathbf{a}_{\alpha})=
  A_{\beta \alpha}\mathbf{a}_{\alpha} + \mathbf{d}_{\beta \alpha} \in \affz (\R^n). $$
  \noindent
  The difference $s_{\beta} - s_{\alpha}$ is given by
  \begin{equation}
    \label{eq:134}
    \hat{\tau}_{\beta \alpha} = \sum\limits_{i=1}^n d^i_{\beta \alpha}
    \frac{\partial}{\partial a^i_{\beta}}. 
  \end{equation}
  \noindent
  In light of equation (\ref{eq:131}), the preimage of
  $\hat{\tau}_{\beta \alpha}$ under the isomorphism of 
  equation (\ref{eq:135}) is given by
  \begin{equation}
    \label{eq:136}
     \sum\limits_{i=1}^n d^i_{\beta \alpha} \de \theta^i_{\beta}.
  \end{equation}
  \noindent
  The cocycle of equation
  (\ref{eq:136}) corresponds to the cocycle $\tau$ defining the
  cohomology class $w_0$ in the proof of Lemma
  \ref{lemma:coho_class_omega_0} and, thus, the result follows.
\end{proof}

Theorem
\ref{thm:equality_forms} allows to use tools from affine geometry to study
problems in the symplectic geometry of Lagrangian bundles and
\textit{vice versa}. For instance, the following theorem holds.

\begin{thm}\label{thm:no_closed_integral_affine_mflds}
  There exist no closed radiant integral affine manifolds.
\end{thm}
\begin{proof}
  Suppose the contrary. Let $\am$ be a closed radiant integral affine manifold
  and let $P_{\am}$ denote the associated period lattice bundle
  (cf. Definition \ref{defn:integral_affine_periods}). Consider the
  symplectic reference bundle
  \begin{equation}
    \label{eq:143}
    \rzn \hookrightarrow (\cotan B/P_{\am}, \omega_0) \to B;
  \end{equation}
  \noindent
  since $B$ and $\rzn$ are closed, so is $\cotan B/P_{\am}$. Therefore
  the cohomology class
  $$ [\omega_0] \in \Hh^2(\cotan B/P_{\am};\R) $$
  \noindent
  is non-zero. Lemma \ref{lemma:coho_class_omega_0} proves that
  $[\omega_0]$ vanishes if and only if $w_0$ vanishes. However,
  Theorem \ref{thm:equality_forms} implies that
  $$ w_0 = \Phi^{-1}(r_{\am}) = 0, $$
  \noindent
  where the second equality follows by assumption. Therefore
  $[\omega_0]=0$, but this is a contradiction.
\end{proof}

\begin{rk}\label{rk:relation_to_literature}
  Theorem \ref{thm:no_closed_integral_affine_mflds} can be used, for instance, to prove that, for $n \geq 2$, $S^n$ is not an integral affine manifold (cf. Example 2.2 and Lemma 2.2 in \cite{sepe_topc}). This is because $\pi_1(S^n) =\{0\}$, which implies that, were $S^n$ to be an integral affine manifold, its radiance obstruction would vanish, which contradicts Theorem \ref{thm:no_closed_integral_affine_mflds}. In fact, the argument of Example 2.2 in \cite{sepe_topc} proves directly that if $S^3$ were an integral affine manifold then its radiance obstruction would vanish by noticing that the cohomology of the total space of the associated symplectic reference bundle (which would necessarily be diffeomorphic to $S^3 \times \R^3/\Z^3$) in degree 2 comes from the inclusion of the fibre $\R^3/\Z^3 \hookrightarrow S^3 \times \R^3/\Z^3$ (which induces an isomorphism in cohomology via the K\"unneth formula).
\end{rk}

\begin{rk} \label{rk:affine_subbundle}
  It is important to notice the difference between the bundle
  $P^* \to B$ and the period lattice bundle $P \to B$ associated to an
  integral affine manifold $\am$. The
  former is an \emph{affine} invariant of $B$ (via the symplectic form
  $\omega_0$), while the latter is only a \emph{linear} invariant,
  since it is the pull-back of a universal lattice defined over
  $\bgl$. The period lattice bundle can be endowed with the structure
  of an \emph{affine lattice} of $\am$ if and only if the radiance
  obstruction $r_{\am}$ is an \emph{integral} form, which, in turn,
  is true if and only if the coordinate changes of the atlas
  $\mathcal{A}$ can be chosen to lie in the group of affine
  transformations of $\Z^n$
  $$ \aff (\Z^n):= \glnz \ltimes \Z^n. $$
  \noindent
  Such manifolds are henceforth called \emph{strongly integral affine
    manifolds}, although it must be noticed that this terminology is
  not standard (cf. \cite{gross_mirror}). In view of Theorem
  \ref{thm:equality_forms}, the symplectic form $\omega_0$ on the symplectic
  reference bundle of a strongly integral affine
  manifold $\am$ is itself integral.
\end{rk}

Combining Theorems \ref{thm:realisability} and \ref{thm:equality_forms}, obtain the main result of the paper.

\begin{mr}\label{cor:10}
  Let $\am$ denote an integral affine manifold with linear holonomy
  $\mathfrak{l}$. An almost Lagrangian bundle $\alf$ over $\am$ is
  Lagrangian if and only if its Chern class $c \in
  \Hh^2(B;\Z^n_{\mathfrak{l}^{-T}})$ satisfies
  \begin{equation}
    \label{eq:148} 
    \theta_B(\Phi^{-1}(r_{\am})) \cdot \psi_B(c^{\R}) = 0,
  \end{equation}
  \noindent
  where the notation is the same as in Corollary \ref{cor:diff_alb}, and
  $r_{\am}$ is the radiance obstruction of the integral affine manifold $\am$.
\end{mr}

\begin{rk}\label{rk:relation_alo}
  Using the fact that the radiance obstruction $r_{\am}$ is the cohomology class defined by the identity map $\mathrm{Id}:\T B \to \T B$ (cf. \cite{goldman}), the above result implies Theorem 3 in \cite{sepe_lag}. This can be seen as follows. The left hand side of equation \eqref{eq:148} equals $\mathcal{D}_{\am}(c)$ by Theorem \ref{thm:realisability}; using \v Cech cocycles for $c$ and $r_{\am}$, it can be checked that the pairing $\theta_B(\Phi^{-1}(r_{\am})) \cdot \psi_B(c^{\R})$ yields a \v Cech cocycle whose corresponding cohomology class in $\Hh^3(B;\R)$ equals the twisted cup product of $c$ as defined in \cite{sepe_lag}, thus yielding the result.
\end{rk}

The Main Result proves that the homomorphism $\mathcal{D}_{\am} $ of
Dazord and Delzant \cite{daz_delz} is completely determined by the integral
affine structure on the base of an almost Lagrangian bundle and by the
universal Chern class $c_U$.
 
\begin{rk}\label{rk:31}
  If $\am$ is a strongly integral affine manifold with linear holonomy
  $\mathfrak{l}$, the Main Result can be strengthened to say that an
  almost Lagrangian bundle over $\am$ is Lagrangian if and only if
  its Chern class $c \in  \Hh^2(B;\Z^n_{\mathfrak{l}^{-T}})$ satisfies
  \begin{equation}
    \label{eq:149}
    \theta(\Phi^{-1}(r_{\am})) \cdot \psi(c) = 0.
  \end{equation}
  \noindent
  In particular, if $\rzn \hookrightarrow (M, \omega) \to B$ is a
  Lagrangian bundle over a strongly integral affine manifold $\am$
  (\textit{i.e.} it induces the affine structure $\mathcal{A}$ on
  $B$), then $\omega$ can always be chosen to be integral. This should
  be compared with Remark 1.2 of \cite{gross_mirror}. Note that it is not
  true that for a fixed integral affine manifold $\am$ there is a
  strongly integral affine manifold $(B',\mathcal{A}')$ in the same
  integral affine diffeomorphism class. This can be seen by
  considering an integral affine two-torus with trivial linear
  holonomy and translational components which are not integral
  (cf. \cite{misha}). 
\end{rk}

The following corollary is a special case of the Main Result.

\begin{cor}\label{cor:11}
  If $\am$ is a radiant affine manifold, then $\mathcal{D}_{\am} = 0$.
\end{cor}

\begin{rk}\label{rk:25}
  Corollary \ref{cor:11} should be compared with what is known in the
  literature regarding exactness of the symplectic form on the total
  space of a Lagrangian (or isotropic) bundle, \textit{e.g.}
  \cite{dui}.
\end{rk}

\subsection{Some examples}\label{sec:examples}
In this section a
manifold is endowed with various radiant integral 
affine structures to illustrate how the classification of Lagrangian
bundles depends on the integral affine geometry of the base
space. \\

The manifold $B = \R^2 \setminus \{\mathbf{0}\}$ inherits an integral affine
structure from $\R^2$ via the natural inclusion
$$ B \hookrightarrow \R^2. $$
\noindent
Denote this integral affine structure by $\mathcal{A}_{0}$. Its
universal cover $\tilde{B}$ can also be endowed with an integral
affine structure $\tilde{\mathcal{A}}_0$; an explicit description of
the affine structure on 
$\tilde{B}$ can be found in \cite{bates_fasso}. It is important to
notice that this affine structure on $\tilde{B}$ is not affinely
isomorphic to the standard affine structure on $\R^2$. For any matrix $A \in
\mathrm{GL}(2;\Z)$, it is possible to define a $\Z$-action on
$(\tilde{B},\tilde{\mathcal{A}}_0)$ which induces an integral affine
structure on $B$ whose affine holonomy is given by the representation
defined on the generator $\gamma$ of $\pi_1(B)$ by
$$ \gamma \mapsto (A,\mathbf{0}). $$
\noindent
For $A_1,~ A_2,~ A_3 \in \mathrm{SL}(2;\Z)$, let
$\mathcal{A}_1,~\mathcal{A}_2,~\mathcal{A}_3$ be the corresponding
radiant integral affine structures on $B$. Consider the integral
affine manifold
\begin{equation}
  \label{eq:150}
  (Y, \mathcal{A}_{Y_3}) = (B, \mathcal{A}_1) \times (B, \mathcal{A}_2)
  \times (B, \mathcal{A}_3). 
\end{equation}
\noindent
This affine manifold is radiant, as it can be seen by looking at its
affine holonomy. Thus $\mathcal{D}_{(Y, \mathcal{A}_Y)} = 0$ by
Corollary \ref{cor:11}. Note that $Y$ has the homotopy type of a
three-torus and so it has $\Hh^3(Y;\R) \cong \R$. This integral affine
manifold therefore provides an example of trivial homomorphism
$\mathcal{D}_{(Y, \mathcal{A}_Y)}$ of
Dazord and Delzant even though its range is non-trivial. \\

Let
$\mathfrak{l}_Y$ denote the linear holonomy of $(Y, \mathcal{A}_Y)$. The
twisted cohomology group 
$$\Hh^2(Y;\Z^6_{\mathfrak{l}^{-T}_Y})$$
\noindent
is
not trivial if and only if at least one of the $A_i$ is
\emph{unipotent}. If this condition is satisfied, then $(Y, \mathcal{A}_Y)$
provides the first example of a manifold whose associated homomorphism
$\mathcal{D}_{(Y, \mathcal{A}_Y)}$ is trivial notwithstanding the fact
that both its domain and range are not trivial. More generally, by
taking the product of $k$ radiant integral affine 
manifolds of the form described above, it is possible to construct
such examples in any even dimension greater than or equal to 6. \\

Consider the product
$$(Z_{\mathbf{n}}, \mathcal{A}_{Z_{\mathbf{n}}}) =
(B,\mathcal{A}_{n_1}) \times \ldots (B,\mathcal{A}_{n_k}),$$
\noindent
where $\mathbf{n} = (n_1,\ldots,n_k) \in \Z_{+}^k$, each $n_j \neq 0$ and
the radiant integral affine structure $\mathcal{A}_{n_j}$ on $B$ has linear
holonomy generated by the matrix
\begin{equation}
  \label{eq:151}
  \begin{pmatrix}
    1 & 0 \\
    -n_j & 1
  \end{pmatrix}.
\end{equation}
\noindent
Let $\mathfrak{l}_{Z_{\mathbf{n}}}$ be the linear holonomy of $(Z_{\mathbf{n}},
\mathcal{A}_{Z_{\mathbf{n}}})$. All
elements of the cohomology group 
$$\Hh^2(Z_{\mathbf{n}};\Z^{2k}_{\mathfrak{l}^{-T}_{Z_{\mathbf{n}}}})$$
\noindent
(which, by the above remark, is non trivial) can be realised as the
Chern class of some regular Lagrangian 
bundle over $(Z_{\mathbf{n}}, \mathcal{A}_{Z_{\mathbf{n}}})$. This
example is interesting because each $(B,\mathcal{A}_{n_j})$ is the
affine model in the neighbourhood 
of a focus-focus singularity of a completely integrable Hamiltonian
system, which is homeomorphic to a two torus
pinched $n_j$ times, as shown in \cite{bates_fasso,zung_ff}. Thus
$(Z_{\mathbf{n}}, \mathcal{A}_{Z_{\mathbf{n}}})$ is a local affine
model for a product of focus-focus singularities. Such products occur
naturally amongst non-degenerate singularities of Lagrangian
fibrations, which have been classified topologically by Zung in
\cite{zung_symp1}. However, it is not known whether there exist
examples of non-degenerate singularities whose regular parts have
non-trivial Chern class and, in particular, whether the Lagrangian
bundles corresponding to non-zero cohomology classes in
$\Hh^2(Z_{\mathbf{n}};\Z^{2k}_{\mathfrak{l}^{-T}_{Z_{\mathbf{n}}}})$
can be compactified to admit (non-degenerate) singularities.


\section{Conclusion}\label{sec:conclusion}
This paper shows that the interplay between the symplectic geometry of
Lagrangian bundles and affine differential geometry runs deep,
allowing to use the methods of one subject to study the other and
\textit{vice versa}. The importance of the radiance obstruction in
constructing Lagrangian fibrations has been recognised by other
authors (\textit{e.g.} Gross and Siebert in \cite{gross_mirror}) who
use different techniques; in particular, the sheaf theoretic approach
developed in \cite{gross_mirror} yields an equivalent proof of the Main
Result. In light of the examples of Section \ref{sec:examples} and given the dearth of explicit examples of
singularities of completely integrable Hamiltonian systems, the following is a natural question to ask.

\begin{qn}
  Are the examples of Section \ref{sec:examples} the regular parts of
  some Lagrangian fibrations? If so, what can be said about the
  topology and symplectic geometry of the singularities?
\end{qn}

Various recent works (\textit{e.g.}
\cite{bolsinov,gross_mirror}) hint at the fact
that there is a strong correlation between the affine geometry of the
regular part of the base space of the fibration and the topology or
symplectic geometry in a neighbourhood of a singularity. This is a
subtle question, since many intervening factors, such as smoothness
conditions, also play a crucial role. By studying integral affine
manifolds with singularities (\textit{e.g.} \cite{gross_mirror}), it
may be possible to generalise the methods of this paper to either
construct such examples explicitly or prove non-existence results,
which would further clarify the nature of Lagrangian fibrations and of
completely integrable Hamiltonian systems.


\bibliographystyle{alpha}
\bibliography{mybib}

\begin{thebibliography}{BVN10}

\bibitem[AM55]{aus}
L.~Auslander and L.~Markus.
\newblock Holonomy of flat affinely connected manifolds.
\newblock {\em Annals of Math.}, 62:139 -- 151, 1955.

\bibitem[Arn78]{arnold}
V.I. Arnold.
\newblock {\em Mathematical Methods in Classical Mechanics}.
\newblock Springer-Verlag, 1978.

\bibitem[Bai01]{baier}
P.~Baier.
\newblock $\mathbb{T}^3$-fibrations on compact six-manifolds.
\newblock {\em arXiv:math/0109087v1}, 2001.

\bibitem[Bat88]{bates_ob}
L.M. Bates.
\newblock Examples for obstructions to action-angle coordinates.
\newblock {\em Proc. Royal Soc. Edin}, 110:27 -- 30, 1988.

\bibitem[Ben60]{benzecri}
J.-P. Benz{\'e}cri.
\newblock Sur les vari\'et\'es localement affines et localement projectives.
\newblock {\em Bull. Soc. Math. France}, 88:229 -- 332, 1960.

\bibitem[BF07]{bates_fasso}
L.M. Bates and F.~Fass{\`o}.
\newblock An affine model for the actions in an integrable system with
  monodromy.
\newblock {\em Regul. Chaotic Dyn.}, 12(6):675 -- 679, 2007.

\bibitem[BS92]{bates_snia}
L.~Bates and J.~Sniatycki.
\newblock On action-angle variables.
\newblock {\em Arch. Rat. Mech. Anal.}, 120:337 -- 343, 1992.

\bibitem[BT99]{bott_tu}
R.~Bott and L.~W. Tu.
\newblock {\em Differential Forms in Algebraic Topology}.
\newblock Springer-Verlag, 1999.

\bibitem[BVN10]{bolsinov}
A.~Bolsinov and S.~Vu~Ngoc.
\newblock Symplectic equivalence for integrable systems with common action
  integrals.
\newblock {\em in preparation}, 2010.

\bibitem[CB97]{cush_dav}
R.H. Cushman and L.M. Bates.
\newblock {\em Global Aspects of Classical Integrable Systems}.
\newblock Birkhauser, 1997.

\bibitem[CV66]{charlap}
L.S. Charlap and A.T. Vasquez.
\newblock The cohomology of group extensions.
\newblock {\em Trans. Amer. Math. Soc.}, 124(1):24 -- 40, 1966.

\bibitem[DD87]{daz_delz}
P.~Dazord and P.~Delzant.
\newblock Le probleme general des variables actions-angles.
\newblock {\em J. Diff. Geom.}, 26:223 -- 251, 1987.

\bibitem[DK01]{davis_kirk}
J.F. Davis and P.~Kirk.
\newblock {\em Lecture Notes in Algebraic Topology}.
\newblock AMS, 2001.

\bibitem[Dui80]{dui}
J.J. Duistermaat.
\newblock On global action-angle coordinates.
\newblock {\em Comm. Pure Appl. Math.}, 33:687 -- 706, 1980.

\bibitem[FGH81]{hirsch}
D.~Fried, W.~M. Goldman, and M.~W. Hirsch.
\newblock Affine manifolds with nilpotent holonomy.
\newblock {\em Comment. Math. Helv.}, 56(4):487 -- 523, 1981.

\bibitem[FS07]{fasso}
F.~Fass{\`o} and N.~Sansonetto.
\newblock Integrable almost-symplectic {H}amiltonian systems.
\newblock {\em J. Math. Phys.}, 48(9), 2007.

\bibitem[Gei92]{geiges}
H.~Geiges.
\newblock Symplectic structures on {$\mathbb{T}^2$}-bundles over
  {$\mathbb{T}^2$}.
\newblock {\em Duke Math Jour.}, 67(3):539 -- 555, 1992.

\bibitem[GH84]{goldman}
W.M. Goldman and M.W. Hirsch.
\newblock The radiance obstruction and parallel forms on affine manifolds.
\newblock {\em Trans. Amer. Math. Soc.}, 286(2):629 -- 649, 1984.

\bibitem[GH86]{goldman_orbits}
W.M. Goldman and M.W. Hirsch.
\newblock Affine manifolds and orbits of algebraic groups.
\newblock {\em Trans. Amer. Math. Soc.}, 295(1):175 -- 198, 1986.

\bibitem[Gro58]{gro}
A.~Grothendieck.
\newblock A general theory of fibre spaces with structure sheaf.
\newblock Technical report, University of Kansas, 1958.

\bibitem[GS06]{gross_mirror}
M.~Gross and B.~Siebert.
\newblock Mirror symmetry via logarithmic degeneration data. {I}.
\newblock {\em J. Differential Geom.}, 72(2):169 -- 338, 2006.

\bibitem[Hus93]{huse}
D.~Husemoller.
\newblock {\em Fibre Bundles}.
\newblock Springer-Verlag, $3^{\mathrm{rd}}$ rev. edition, 1993.

\bibitem[Luk08]{luk}
O.~Lukina.
\newblock {\em Geometry of torus bundles in integrable Hamiltonian systems}.
\newblock PhD thesis, University of Groningen, 2008.

\bibitem[McC85]{mccleary}
J.~McCleary.
\newblock {\em A User's Guide to Spectral Sequences}.
\newblock Publish or Perish, 1985.

\bibitem[Mil58]{milnor_euler}
J.~Milnor.
\newblock On the existence of a connection with curvature zero.
\newblock {\em Comment. Math. Helv.}, 32:215--223, 1958.

\bibitem[Mis96]{misha}
K.N. Mishachev.
\newblock The classification of {L}agrangian bundles over surfaces.
\newblock {\em Diff. Geom. Appl.}, 6:301 -- 320, 1996.

\bibitem[MM74]{markus_meyer}
L.~Markus and K.~R. Meyer.
\newblock {\em Generic {H}amiltonian dynamical systems are neither integrable
  nor ergodic}.
\newblock American Mathematical Society, Providence, R.I., 1974.
\newblock Memoirs of the American Mathematical Society, No. 144.

\bibitem[Ngo03]{vu_ngoc}
S.~Vu Ngoc.
\newblock On semi-global invariants for focus-focus singularities.
\newblock {\em Topology}, 42(2):365 -- 380, 2003.

\bibitem[Sep10a]{sepe_klein}
D.~Sepe.
\newblock Classification of {L}agrangian fibrations over a {K}lein bottle.
\newblock {\em Geom. Dedicata}, 149:347 -- 362, 2010.

\bibitem[Sep10b]{sepe_topc}
D.~Sepe.
\newblock Topological classification of {L}agrangian fibrations.
\newblock {\em Jour. Geom. Phy.}, 60:341 -- 351, 2010.

\bibitem[Sep11]{sepe_lag}
D.~Sepe.
\newblock Almost {L}agrangian obstruction.
\newblock {\em Differential Geom. Appl.}, 29(6):787--800, 2011.

\bibitem[Sie67]{siegel}
J.~Siegel.
\newblock Higher order cohomology operations in local coefficient theory.
\newblock {\em Amer. J. Math.}, 89:909 -- 931, 1967.

\bibitem[Smi81]{smillie}
J.~Smillie.
\newblock An obstruction to the existence of affine structures.
\newblock {\em Inventiones Mathematicae}, 64:411 -- 415, 1981.

\bibitem[SS11]{sans_sepe}
N.~Sansonetto and D.~Sepe.
\newblock Almost complex realisations of {P}oisson manifolds.
\newblock {\em in preparation}, 2011.

\bibitem[Vai94]{vaisman}
I.~Vaisman.
\newblock {\em Lectures on the geometry of {P}oisson manifolds}, volume 118 of
  {\em Progress in Mathematics}.
\newblock Birkh\"auser Verlag, Basel, 1994.

\bibitem[Wei71]{weinstein}
A.~Weinstein.
\newblock Symplectic manifolds and their {L}agrangian submanifolds.
\newblock {\em Advances in Math.}, 6:329 -- 346, 1971.

\bibitem[Whi78]{white}
G.~W. Whitehead.
\newblock {\em Elements of homotopy theory}.
\newblock Springer-Verlag, 1978.

\bibitem[Zun96]{zung_symp1}
N.T. Zung.
\newblock Symplectic topology of integrable {H}amiltonian systems i:
  Arnold-{L}iouville with singularities.
\newblock {\em Compositio Math.}, 101:179 -- 215, 1996.

\bibitem[Zun97]{zung_ff}
N.T. Zung.
\newblock A note on focus-focus singularities.
\newblock {\em Diff. Geom. and Appl.}, 7:123 -- 130, 1997.

\bibitem[Zun03]{zun_symp2}
N.T. Zung.
\newblock Symplectic topology of integrable {H}amiltonian systems ii:
  Topological classification.
\newblock {\em Compositio Math.}, 138:125 -- 156, 2003.

\end{thebibliography}
\end{document}